\documentclass{amsart}

\usepackage[english]{babel}
\usepackage{hyperref}
\usepackage{enumitem}
\usepackage{graphicx}
\usepackage{xcolor} 

\newtheorem{theorem}{Theorem}[section]
\newtheorem{definition}[theorem]{Definition}
\newtheorem{proposition}[theorem]{Proposition}
\newtheorem{corollary}[theorem]{Corollary}
\newtheorem{lemma}[theorem]{Lemma}
\newtheorem{example}[theorem]{Example}
\newtheorem{step}{Step}
\theoremstyle{remark}
\newtheorem{remark}[theorem]{Remark}

\numberwithin{equation}{section}

\DeclareMathOperator{\f}{\mathcal{F}} 
\DeclareMathOperator{\bb}{B} 
\DeclareMathOperator{\ga}{\mathsf{a}} 
\DeclareMathOperator{\g}{\mathsf{g}} 
\DeclareMathOperator{\tub}{\mathrm{Tub}} 
\DeclareMathOperator{\vt}{\mathcal{V}} 
\DeclareMathOperator{\hz}{\mathcal{H}} 
\DeclareMathOperator{\e}{\mathbf{e}} 
\DeclareMathOperator{\vol}{vol} 

\begin{document}

\title[Foliations of $(\alpha,\beta)$-spaces]{On Singular Finsler Foliations of $(\alpha,\beta)$-spaces}

\author[M. M. Alexandrino]{Marcos M. Alexandrino}
\address{Marcos M. Alexandrino \hfill\break\indent Instituto de Matem\'{a}tica e Estat\'{\i}stica, Universidade de S\~ao Paulo \hfill\break\indent Rua do Mat\~ao, 1010, 05508-090 S\~ao Paulo, S\~ao Paulo, Brazil}
\email{malex@ime.usp.br}

\thanks{This project was partially funded by FAPESP (S\~ao Paulo Research Foundation); namely, Alexandrino is supported by grant No. 2022/16097-2, and Mar\c cal by grant No. 2024/03446-4.}

\author[B. O. Alves]{Benigno O. Alves}
\address{Benigno O. Alves \textup{(corresponding author)} \hfill\break\indent Instituto de Matem\'{a}tica e Estat\'{\i}stica, Universidade Federal da Bahia \hfill\break\indent Rua Bar\~ao de Jeremoabo, 40170-115 Salvador, Bahia, Brazil}
\email{benignoalves@ufba.br}

\author[P. Mar\c cal]{Patr\'icia Mar\c cal}
\address{Patr\'icia Mar\c cal \hfill\break\indent Instituto de Matem\'{a}tica e Estat\'{\i}stica, Universidade de S\~ao Paulo \hfill\break\indent Rua do Mat\~ao, 1010, 05508-090 S\~ao Paulo, S\~ao Paulo, Brazil}
\email{pmarcal@usp.br}

\date{\today}


\begin{abstract}
We investigate singular Finsler foliations (SFFs) on a manifold equipped with an $(\alpha,\beta)$-metric. To be precise, we verify that any SFF of an $(\alpha, \beta)$-space is, under some hypotheses on the metric, a singular Riemannian foliation (SRF). This gives a partial answer to the general question ``under which conditions a SFF is a SRF with respect to some Riemannian metric''. Moreover, we extend the proof of Molino's conjecture to SFFs whenever they are also a SRFs. Finally, we prove equifocality of the regular leaves for a SFF under the same condition.
\end{abstract}

\maketitle
\bibliographystyle{abbrv}


\section{Introduction}

Finsler metrics on a smooth manifold $M$ correspond to positive Lagrangians on $TM$ that are smooth away from zero section, positively homogeneous of degree one, and whose fundamental tensor is positive definite. They include the norm associated to any Riemannian metric. However, while Riemannian metrics model the dynamics of isotropic conservative mechanical systems \cite{Jacobi1837reduction, Godinho2014introduction}, Finsler metrics have the potential to describe anisotropic systems, with applications to several areas such as physics, biology, and robotics \cite{Javaloyes2023general, Ratliff2021generalized, Brody2015solution, Yajima2009finsler, Antonelli1993theory}.

Among the most widely studied examples of Finsler metrics are $(\alpha,\beta)$-metrics, i.e. those that can be written as
\begin{equation*}
F=\alpha\,\phi\!\left(\frac{\beta}{\alpha}\right) ,
\end{equation*}
for a Riemannian norm $\alpha$, a $1$-form $\beta$, and some positive smooth function $\phi$. They stand out for their computability, and can be geometrically characterized as the most symmetric non-Riemannian Finsler metrics (pointwise); see \cite[Theorem~2.2]{Yu2011new} or Proposition~\ref{AlphaBetaSymmetry}. In fact, the isometries of $(\alpha,\beta)$-spaces have been classified; see \cite{Voicu2023finsler} and references therein. By means of a coordinate-free approach, we recover the classification of isometries for $(\alpha,\beta)$-spaces under the assumption that $\phi^{\prime}$ never vanishes; vide Corollary~\ref{AlphaBetaIsometry}. The key point, from Proposition~\ref{AlphaBeta!Phi}, is that an $(\alpha,\beta)$-metric is uniquely determined by the pair $(\alpha,\beta)$  when $\phi$ is strictly monotone.

The class of $(\alpha, \beta)$-metrics trivially includes Riemannian norms (e.g. $\phi$ constant), but it encompasses several other important families of Finsler metrics, for instance, Randers metrics \cite{Randers1941asymmetrical}, and Matsumoto (or slope) metrics \cite{Matsumoto1989slope}. Historically, $(\alpha, \beta)$-metrics were proposed by Matsumoto as direct generalization of Randers metrics \cite{Matsumoto1972creducible}, which are considered the simplest non-Riemannian Finsler metrics. These are of the form $R = \alpha +\beta$, and their geodesics naturally appear as solutions of the \emph{Zermelo navigation problem} \cite{Bao2004zermelo}. In particular, the navigation model for Randers spaces has been related to the modeling of wildfire spread \cite{Markvorsen2016finsler}. In addition, the influence of the slope on fire-growth can be modeled by a Matsumoto metric \cite{Javaloyes2023general}.

These geometric and physical motivations led us to examine the global geometry, in which foliation theory plays a central role; thereupon our main subject. Besides, foliations provide a useful framework for studying wave fronts. Case in point, a wave propagating at constant speed along its fronts in a Riemann-Finsler space is governed by an isoparametric function, whose level sets form a foliation with leaves of constant mean curvature; refer to \cite{Alves2024isoparametric} and bibliographical sources cited within.

Singular Riemannian foliations (SRFs) encode many symmetry phenomena: they unify isometric group actions, Riemannian submersions, and partitions of a fiber bundle by the action of an holonomy groupoid of a metric connection. Their analogues, singular Finsler foliations (SFFs), are groundwork for orbits of Finsler actions, fibers of Finsler submersions, and level sets of transnormal functions \cite{Alexandrino2019singular}.

A SFF $\{ L_p \}_{p \in M}$ on a (complete) Finsler manifold $(M,F)$ is a singular foliation whose leaves are locally equidistant, i.e. geodesics starting orthogonally at one leaf must be orthogonal to all leaves they meet. Associated with the question posed by Ghys in \cite[Appendix~E]{Molino1988riemannian}, there is the problem to determine whether a singular Finsler foliation (SFF) is a singular Riemannian foliation (SRF) with respect to some Riemannian metric. The affirmative answer has already been given for regular foliations \cite{Jozefowicz2008finsler, Miernowski2006lift, Popescu2019lagrangians} and for SFFs of Randers spaces \cite{Alexandrino2019singular}. In the present article, we characterize SFFs on $(\alpha,\beta)$-spaces such that the reply is positive for the Riemannian norm $\alpha$; c.f. \cite[Theorem~1.1]{Alexandrino2019singular}. Moreover, we find sufficient conditions for the same response with respect to a Riemannian norm conformal to $\alpha$.

\begin{theorem}
\label{AlphaBetaSFF}
Let $\f$ be a singular Finsler foliation with (locally) closed leaves of an $(\alpha,\beta)$-space $(M,F)$, where $F = \alpha \phi \big( \tfrac{\beta}{\alpha} \big)$ for a Riemannian norm $\alpha(\cdot) = \sqrt{\ga(\cdot,\cdot)}$ and a $1$-form $\beta$ on $M$. Let $\bb$ denote the $\alpha$-dual vector field to $\beta$.
\begin{enumerate}[label=\textup{(\roman*)}]
\item\label{f1} If $\phi: (-b_0 , b_0) \to \mathbb R$ has nowhere zero derivative and $\bb$ is $\alpha$-orthogonal to $\f$, then $\mathcal F$ is a SRF with respect to $\ga$, and the vector field $\bb$ is foliated (i.e. its flow locally maps one leaf to another) and horizontal (i.e. $F$-orthogonal to the leaves).

\item\label{f2} If $\bb$ is never zero and tangent to the leaves of $\mathcal F$, then there exists a smooth function $\kappa: M\to \mathbb R$ such that $\mathcal F$ is SRF with respect to $\kappa\ga$.
\end{enumerate}
\end{theorem}

The proof of Theorem \ref{AlphaBetaSFF} is based on understanding 
Finsler submersions of $(\alpha,\beta)$-spaces; vide Proposition~\ref{AlphaBetaSubm}.

\begin{remark}
\label{TangencyToStrata}
Under the hypothesis of nonvanishing $\phi^{\prime}$ alone, the vector field $\bb$ is necessarily tangent to the strata of $\f$; from Lemma~\ref{LinearTangencyMinStrat} and the Slice Theorem~\ref{SliceThm}.
\end{remark}

We now disclose an example on how item~\ref{f1} of Theorem~\ref{AlphaBetaSFF} fails when the nonvaninshing of $\phi^{\prime}$ is dropped as hypothesis.

\begin{example}
\label{AlphaBetaSFnotR}
Let $\alpha$ be a Euclidean norm and $\beta$ a linear form on $\mathbb R^n$. If $\alpha^{\ast}(\beta) < 1$, then $F = \tfrac{\alpha^2 + \beta^2}{\alpha}$ is a well-defined $(\alpha,\beta)$-metric on $\mathbb R^n$, where $\phi(s) = s^2 + 1$. Once the metric is reversible, the set of indicatrices, $\f = \{ \mathcal I^F (r) \, \vert \, r \geq 0 \}$, is a singular Finsler foliation for $(\mathbb R^n,F)$. Whenever $\beta$ is nonvanishing, its $\alpha$-dual vector $\bb$ fails to be tangent to the minimum stratum, which is a point (the origin). In this case, the metric $F$ is not Riemannian and, moreover, the indicatrix $\mathcal I^F (1)$ is not quadratic. Hence, $\f$ is not a singular Riemannian foliation for $(\mathbb R^n,\alpha)$.
\end{example}

Next, we should stress that a converse for Theorem~\ref{AlphaBetaSFF}, item~\ref{f1}, is true.

\begin{remark}
\label{AlphaBetaConverse}
Consider a Riemannian manifold $(M,\g)$.
\begin{enumerate}[label=\textup{(\roman*)}]
\item Let $\f$ be a SRF on $(M,\g)$, and

\item \label{PhiCond} let $\phi: (-b_0 , b_0) \to \mathbb R$ be a positive smooth function that has nowhere zero derivative and satisfies 
\begin{equation*}
(\phi(s) - s\phi^{\prime}(s)) + (b^2 - s^2)\phi^{\prime\prime}(s) > 0 , 
\end{equation*}
for all $s$ and $b$ with $\vert s \vert \leq b < b_0$.

\item Suppose $\bb$ is a foliated vector field $\g$-orthogonal to $\f$ with $\sqrt{\g(\bb,\bb)} < b_0$.
\end{enumerate}
Set  $\alpha(\cdot) := \sqrt{\g(\cdot,\cdot)}$ and $\beta(\cdot) := \g(\bb,\cdot)$. Then, $\f$ is a SFF with respect to the $(\alpha,\beta)$-metric $F = \alpha \phi \big( \tfrac{\beta}{\alpha} \big)$.
\end{remark}

The observation above allows us to directly adapt the argument presented in \cite{Alexandrino2019singular} to construct infinitely many non-homogeneous SFFs with respect to an $(\alpha,\beta)$-metric on the sphere. For clarity, let us recall the argument here.

\begin{example}
\label{SphereSFF}
In \cite{Radeschi2014clifford}, Radeschi constructed a polynomial map $\pi_{C}:\mathbb{S}^{2l-1}\to \mathbb{R}^{m+1}$ such that, for $l > m+1$ and $m \neq 1,2,4$, the fibers of $\pi_{C}$ are (non-homogeneous)
leaves of a SRF $\f$ on the round sphere $\mathbb{S}^{2l-1}$, whose leaf space is the disk $\mathbb{D}_{C}$. Let $\phi$ be a positive smooth function that satisfies the conditions of item~\ref{PhiCond} in Remark~\ref{AlphaBetaConverse}, and $\bar{\bb} \in \mathfrak{X}(\mathbb{D}_{C})$ a small radial vector field on the disk $\mathbb{D}_{C}$ that is zero at the center and zero near the border $\partial \mathbb{D}_{C}$. Consider its orthogonal lifted vector field $\bb$, i.e. the orthogonal foliated vector field on $\mathbb{S}^{ 2 l-1}$ that projects to $\bar{\bb}$. It follows from Remark~\ref{AlphaBetaConverse} that $\f$ is SFF with respect to the Finsler norm $F=\alpha \phi \big( \tfrac{\beta}{\alpha} \big)$, where $\beta$ is the (Riemannian) dual of $\bb$.
\end{example}

SRFs possess several interesting dynamical and topological properties. At the present, we concern ourselves with the following: \emph{closing the leaves of a SRF results in a partition that is still a SRF}. In other words, if $\f = \{ L \}$ is a SRF, then the partition $\overline{\f} = \{ \overline{L} \}$, given by the closures of the leaves of $\f$, is again a SRF. This property was described in the theory of SRFs as \emph{Molino's conjecture} \cite{Molino1988riemannian}, and proved in \cite{Alexandrino2017closure}. It is not only elegant, and helpful in understanding the dynamics of a foliation, but also useful in analysis. As a matter of fact, the basic functions (i.e. smooth functions that are constant on the leaves) of $\f$ coincide with the basic functions of $\overline{\f}$. In general, it is much easier to deal with a foliation whose leaves are closed, among other reasons, because the leaf space is Hausdorff. The proof of Molino's conjecture \cite{Alexandrino2017closure} yields the next result on singular Finsler foliations.

\begin{proposition}
\label{MolinoSFF}
Let $\f=\{L\}$ be a singular Finsler foliation with locally closed leaves  on a complete Finsler manifold $(M,F)$. If $\f$ is a singular Riemannian foliation with respect to a complete Riemannian norm $\alpha$ on $M$, then the partition by closures $\overline{\f}=\{\overline{L}\}$ is a singular Finsler foliation.
\end{proposition}

The above proposition and Theorem\ref{AlphaBetaSFF} imply our second result on singular foliations of $(\alpha,\beta)$-spaces.

\begin{corollary}
Let $\f$ be a singular Finsler foliation with locally closed leaves of a complete $(\alpha,\beta)$-space $(M,F)$, where $F = \alpha \phi \big( \tfrac{\beta}{\alpha} \big)$ for a complete Riemannian norm $\alpha$ and a $1$-form $\beta$ on $M$.
Assume one of the following conditions:
\begin{enumerate}[label=\textup{(\roman*)}]
\item $\phi$ has nowhere zero derivative, and $\bb$ is $\alpha$-horizontal to the leaves of $\f$; or

\item $\bb$ is vertical (nonvanishing vector field).
\end{enumerate}
Then $\overline{\f}=\{\overline{L}\}$ is again a singular Finsler foliation.
\end{corollary}

Another property satisfied by the regular leaves (i.e. leaves of maximum dimension) of an SRF is \emph{equifocality}. This concept was introduced by Terng and Thorbergsson in \cite{Terng1995submanifold}. Roughly speaking, when we push a submanifold along its normal directions by the Riemannian  exponential, all resulting “parallel sets” remain (immersed) submanifolds, even across focal points. Equifocality has proved fundamental in understanding the topology, metric structure, and dynamical behavior of SRFs; e.g. \cite{Alexandrino2010desingularization, Alexandrino2017closure, Alexandrino2017smoothness}. The notion can be easily extended to the Finslerian setting \cite{Alexandrino2019singular, Alexandrino2024equifocal}.

\begin{definition}
\label{DefEquifocal}
Let $\f = \{ L_p \}_{p \in M}$ be a singular Finsler foliation on a complete Finsler manifold $(M,F)$. A regular leaf $L$ of $\f$ is called \textbf{equifocal} if every point $q \in L$ admits a neighborhood $P_{q}$ of $q$ in $L$ such that for each foliated horizontal unit vector field $\xi$ along $P_{q}$ and for each $r>0$ we have:
\begin{enumerate}[label=\textup{(\roman*)}]
\item the rank of the derivative of $\eta_{r\xi}: P_q \to M$ is constant, where the \textbf{endpoint map} $\eta_{r\xi}$ is defined as $\eta_{r\xi}(x)=\exp_{x}(r\xi)$; and

\item $\eta_{r\xi}(P_{q}) \subset L_{\eta_{r\xi}(q)}$. 
\end{enumerate}
\end{definition}

In what follows, we continue to investigate the concept of equifocality for regular leaves of a SFF, which began in \cite[Corollary~1.2]{Alexandrino2019singular} and \cite[Theorem~1.1]{Alexandrino2024equifocal}.

\begin{theorem}
\label{FinslerEquifocal}
Let $\f = \{ L \}$ be a singular Finsler foliation on a complete Finsler manifold $(M,F)$. Assume that:
\begin{enumerate}[label=\textup{(\roman*)}]
\item there exists a complete Riemannian norm $\alpha$ on $M$ such that $\f$ is a singular Riemannian foliation with respect to $\alpha$; and

\item $\f$ has (locally) closed leaves.
\end{enumerate}
Then the regular leaves of $\f$ are equifocal submanifolds with respect to the Finsler metric $F$.
\end{theorem}

The previous result guarantees the equifocality of regular leaves on a SFF under conditions that are valid for many $(\alpha,\beta)$-spaces.

\begin{corollary}
\label{AlphaBetaEquifocal}
Let $\f$ be a singular Finsler foliation with (locally) closed leaves of a complete $(\alpha,\beta)$-space $(M,F)$, where $F = \alpha \phi \big( \tfrac{\beta}{\alpha} \big)$ for a complete Riemannian norm $\alpha$ and a $1$-form $\beta$ on $M$, and let $\bb$ denote the $\alpha$-dual to $\beta$. Assume one of the following conditions:
\begin{enumerate}[label=\textup{(\roman*)}]
\item $\phi$ has nowhere zero derivative, and $\bb$ is $\alpha$-horizontal to the leaves of $\f$; or

\item $\bb$ is a vertical nonvanishing vector field.
\end{enumerate}
Then the regular leaves of $\f$ are equifocal with respect to the metric $F$. 
\end{corollary}

This paper is organized as follows. We begin by reviewing some fundamental facts of Finsler geometry in Section~2. In the subsequent Section~3, we highlight some characteristics of $(\alpha,\beta)$-metrics, which will be needed later on. In Section~4, we briefly recall a few properties of Finsler submersions, based on \cite{Alvarez2001isometric, Alexandrino2019singular}. Section~5 is then dedicated to the proof of Proposition~\ref{AlphaBetaSubm}, on $(\alpha,\beta)$-submersions. In what follows, Section~6 holds some results on singular Finsler foliations from the article \cite{Alexandrino2019singular}; roughly speaking, they generalize classical results on proper actions; e.g. \cite{Alexandrino2015lie}. Sections~7 and~8 present the proofs of Theorems~\ref{AlphaBetaSFF} and~\ref{FinslerEquifocal}, respectively. At last, Section~9 contains the demonstration of Proposition~\ref{MolinoSFF}.
\section{Finsler metrics}

Let $V$ be an $n$-dimensional real vector space. A \textbf{Minkowski norm} on $V$ is a continuous function $F: V \to [0, \infty)$ such that
\begin{enumerate}[label=\textup{(\roman*)}]
\item $F$ is smooth on $V \setminus 0$;

\item  $F$ is positively homogeneous of degree one, i.e. $F(\lambda v) = \lambda F(v)$, for all $v \in V$ and $\lambda > 0$; and

\item  for each $v\in V \setminus 0$, the \textbf{fundamental tensor}  of $F$ at $v$, defined as
\begin{equation}
\g^F_v (u,w) := \left. \frac{\partial^2}{\partial s \partial t} \left[ \frac{1}{2} F^2(v + su + tw) \right] \right\vert_{s=t=0} , \text{ for all } u,w \in V ,
\end{equation}
is a positive-definite symmetric bilinear form.
\end{enumerate}
The pair $(V, F)$ is said to be a \textbf{Minkowski space}.

Let $M$ be an $n$-dimensional manifold. A continuous function $F:TM\to [0,\infty)$ is a \textbf{Finsler metric} if
\begin{enumerate}[label=\textup{(\roman*)}]
\item $F$ is smooth on $TM\setminus 0$; and

\item for each $p\in M$, $F_p := F\vert_{T_p M}$ is a Minkowski norm on $T_p M$. 
\end{enumerate}
The pair $(M,F)$ is called a \textbf{Finsler manifold}. Notice that $F$ naturally induces many Riemannian metrics on $T_p M$ by $\tilde{\g}_v = \g^F_v$, for each $v \in T_p M \setminus 0$.

\begin{example}
A Riemannian norm $\alpha$ is a Finsler metric whose associated fundamental tensor coincides with the Riemannian metric defined by $\alpha$, i.e. $\g^\alpha = \ga$ for $\alpha(\cdot) = \sqrt{\ga(\cdot,\cdot)}$. Define the Lagrangian $R = \alpha + \beta$, where $\beta$ is a 1-form. Then $R$ is a Finsler metric if and only if the dual norm of $\beta$ satisfies $\alpha^{\ast} (\beta) < 1$. In this case, the Finsler metric $R$ is called a \textbf{Randers metric} with data $(\alpha, \beta)$.
\end{example}

\begin{lemma}\cite[Lemma~2.1]{Alexandrino2019singular}
\label{FundTensor}
Let $(M, F)$ be a Finsler manifold. Given $v \in TM \setminus 0$, the fundamental tensor of $F$ satisfies:
\begin{enumerate}[label=\textup{(\roman*)}]
    \item $\g^F_{\lambda v} = \g^F_v$ for all $\lambda > 0$;
    \item $\g^F_v(v, v) = F^2(v)$; and
    \item $\g^F_v(v, w) =  \left. \frac{\partial}{\partial t} \left[ \frac{1}{2}F^2(v + tw) \right] \right\vert_{t=0}$ for all $w \in T_{\tau(v)} M$.
\end{enumerate}
Here and throughout, $\tau: TM \to M$ denotes the canonical projection.
\end{lemma}

In words, the second item of Lemma~\ref{FundTensor} says a Finsler metric $F$ can be recovered from its fundamental tensor. So $F$ can be interpreted as a family of inner products that depends on the direction.

The \textbf{indicatrix} of a Finsler metric $F$ is the hypersurface $\mathcal I^F := F^{-1}(1) \subset TM$. The \textbf{indicatrix at $p \in M$} is the hypersurface in $T_pM$ defined by 
\begin{equation*}
\mathcal{I}_p^F := \{ v \in T_p M : F(v) = 1 \} .
\end{equation*}
It can be verified that $\mathcal{I}_p^ F$ is strongly convex, diffeomorphic to the standard unit sphere, the origin is contained in the region bounded by $\mathcal{I}_p^F$, and each ray emanating from the origin intersects $\mathcal{I}_p^F$ at exactly one point.  Furthermore, for any $v \in \mathcal{I}_p^F$, the tangent space $T_v\mathcal{I}_p^F$ coincides with the subspace $\{w \in T_pM : \g^F_v(v, w) = 0\}$, which leads to the definition of orthogonality.

Given a submanifold $P$ of a Finsler manifold $(M, F)$, we say that a vector $v \in T_q M$ is \textbf{orthogonal to $P$} at $q \in P$ if $\g^F_v (v,w) = 0$ for all $w \in T_q P$. The set of all (non-zero) orthogonal vectors to $P$ at $q$, denoted $\nu_q(P)$, is called an \textbf{orthogonal cone}. As the name suggests, $\nu_q(P)$ is not necessarily a subspace, but a cone in $T_q M$, i.e. $\lambda v \in \nu_q (P)$ for all $v \in \nu_q (P)$ and $\lambda >0$. Moreover, one can prove that $\nu_q(P)$ is a submanifold of $T_q M$ with dimension $n - \dim P$. A similar result holds for the set of all orthogonal vectors to $P$, denoted $\nu(P)$, and the space of unit orthogonal vectors to $P$, $\nu^1(P)$.

\begin{proposition}\cite[Proposition~2.3]{Alexandrino2019singular}
Let $(M, F)$ be a Finsler manifold, and $P \subset M$ a submanifold. Then the orthogonal cone $\nu(P)$ and the unit orthogonal cone $\nu^1(P)$ are smooth submanifolds of $TM$. Furthermore, the restrictions of the canonical projection $\tau : \nu(P) \to P$ and $\tau : \nu^1(P) \to P$ are submersions.
\end{proposition}

The \textbf{geodesics} of the Finsler metric $F$ are precisely the geodesics of its energy Lagrangian $\tfrac{1}{2} F^2$, i.e. the critical points of the energy functional, which associates to each piecewise smooth curve $\gamma: [a, b] \to M$ the quantity
\begin{equation*}
E(\gamma) =  \frac{1}{2} \int_{a}^{b} F^2(\gamma^{\prime}(t)) \, dt .
\end{equation*}
By the regularity of $F^2$ (to be precise, smooth on $TM \setminus 0$ and, generally, no more than $C^1$ along the zero section) and strong convexity (assured by the positive-definiteness of the fundamental tensor), existence and uniqueness of geodesics for given initial conditions are guaranteed. So one can define the \textbf{exponential map} on an open subset $\mathcal{U} \subset TM$, for those vectors $v$ such that the maximal interval of existence, $I_v$,  of the geodesic $\gamma_v$ satisfying $\gamma_v^{\prime}(0) = v$, contains the value $1$. Indeed, the map $\exp^F: \mathcal{U} \to M$ is given by $\exp^F(v) := \gamma_v(1)$. It is $C^1$ over the zero section and smooth away from there.

We call attention to the fact that the geodesics $\gamma_v$ and $\gamma_{-v}$ are not necessarily related, because $F$ need not be reversible, i.e. $F(v) \neq F(-v)$. In consequence, it is possible for a geodesic $\gamma_v$ to have a maximal interval of existence $I_v$ containing the rays $[0,\infty)$ or $(-\infty,0]$ independently; e.g. \cite[\S~12.6]{Bao2000introduction}. Hence, we say the Finsler manifold $(M,F)$ is \textbf{forward (resp. backward) complete} if every geodesic can be forward (resp. backward) extended indefinitely. When $(M,F)$ is both forward and backward complete, it is just called \textbf{complete}.

The \textbf{length} of a piecewise smooth curve $\gamma: [a, b] \to M$ is given by the integral of $F$ along the curve:
\begin{equation*}
L(\gamma) = \int_a^b F(\gamma^{\prime}(t)) \, dt .
\end{equation*}
This functional is invariant under positive reparameterizations due to the positive 1-homogeneity of $F$. So the \textbf{distance} between two submanifolds, $P$ and $N$, is defined as
\begin{equation*}
d(P, N) = \inf \{ L(\gamma) : \gamma(a) \in P, \gamma(b) \in N \} .
\end{equation*}
Particularly, the distance between two points is a distance function in the usual sense except for reversibility, i.e. $d(p,q) \neq d(q,p)$. Nonetheless, the concepts of geodesic and metric completeness are equivalent. Geodesics minimize the length over sufficiently small intervals. More generally, a geodesic $\gamma$ (locally) minimizes the distance from a submanifold $P$ to a point $q = \gamma(b)$ if and only if $\gamma$ is $F$-orthogonal to $P$ at its starting point $\gamma(a) \in P$.
\section{\texorpdfstring{$(\alpha,\beta)$}{(alpha,beta)}-metrics}

\subsection{Characterizations}

A Minkowski norm $F$ on $V$ will be called an \textbf{$(\alpha,\beta)$-norm} if it takes the form $F = \alpha \phi \big( \tfrac{\beta}{\alpha} \big)$ for a Euclidean norm $\alpha$, a linear functional $\beta$, and a real function $\phi : (-b_0,b_0) \to \mathbb{R}$ with $\alpha^{\ast}(\beta) \in [0,b_0)$.

The $(\alpha,\beta)$-norms are the most symmetric non-Euclidean Minkowski norms. Yu and Zhu make this geometric meaning explicit in \cite[Theorem~2.2]{Yu2011new}. Here, we show an alternative proof for vector spaces of dimension $n \geq 2$. The $1$-dimensional case is trivial by \cite[Remark~3]{Yu2011new}.

\begin{proposition}
\label{AlphaBetaSymmetry}
Let $V$ be a real vector space of dimension $n \geq 2$. Let $\alpha$ be a Euclidean norm, and $\beta$ a linear functional on $V$. Suppose $F$ is a continuous and positively $1$-homogeneous function on $V$ that is positive of class $C^k$ on $V \setminus 0$. If $F$ is invariant under the group
\begin{equation*}
\mathbb{O}_{\beta}(\alpha) = \{ Q \in \mathbb O(\alpha): \beta \circ Q = \beta\},
\end{equation*}
then there exists a positive $C^k$ function $\phi: (-b_0, b_0) \to \mathbb{R}$ such that
\begin{equation*}
F = \alpha \phi \big( \tfrac{\beta}{\alpha} \big) ,
\end{equation*}
where $ \alpha^{\ast}(\beta) \in [0,b_0)$. In particular, the indicatrix $\mathcal I^F$ is a hypersurface of revolution around the axis defined by $\bb$, the $\alpha$-dual of $\beta$. Conversely, if $F = \alpha \phi\big( \tfrac{\beta}{\alpha} \big)$ for some function $\phi$ (as above), then $F$ is invariant under $\mathbb{O}_{\beta}(\alpha)$.
\end{proposition}
\begin{proof}
For each $s \in [-\alpha^{\ast}(\beta), \alpha^{\ast}(\beta)]$, there exists $v \in V \setminus 0$ such that $\tfrac{\beta(v)}{\alpha(v)} = s$, because $\dim V = n \geq 2$. Define
\begin{equation*}
\phi(s) := \frac{F(v)}{\alpha(v)} .
\end{equation*}
It suffices to show that $\phi$ is well-defined. Let $w \in V \setminus 0$ be any vector such that $\tfrac{\beta(w)}{\alpha(w)} = \tfrac{\beta(v)}{\alpha(v)}$. Since both, $\tfrac{v}{\alpha(v)}$ and $\tfrac{w}{\alpha(w)}$, lie on the $\alpha$-unit sphere and have the same $\beta$-value, there exists an isometry $Q \in \mathbb{O}_{\beta}(\alpha)$ such that $Q \left( \tfrac{v}{\alpha(v)} \right) = \tfrac{w}{\alpha(w)}$. Using positive 1-homogeneity and $\mathbb{O}_{\beta}(\alpha)$-invariance of $F$, we have that
\begin{equation*}
\frac{F(w)}{\alpha(w)} = F \left( \frac{w}{\alpha(w)} \right) = F\left(Q\left(\frac{v}{\alpha(v)}\right)\right) = F\left(\frac{v}{\alpha(v)}\right) = \frac{F(v)}{\alpha(v)},
\end{equation*}
which proves that $\phi$ is well-defined.

Conversely, if $F = \alpha \phi \big( \tfrac{\beta}{\alpha} \big)$, then for any $Q \in \mathbb{O}_{\beta}(\alpha)$, we have $\alpha(Qv) = \alpha(v)$ and $\beta(Qv) = \beta(v)$. Thus,
\begin{equation*}
F(Q v) = \alpha(Qv) \phi \!\left( \frac{\beta(Qv)}{\alpha(Qv)} \right) = \alpha(v) \phi \!\left( \frac{\beta(v)}{\alpha(v)} \right) = F(v),
\end{equation*}
completing the proof.
\end{proof}

Next, a computational characterization for $(\alpha,\beta)$-norms.

\begin{lemma}\cite[Lemma~1.1.2]{Chern2005riemann}
On a vector space $V$, let $\alpha$ be a Euclidean norm and $\beta$ a linear functional. Let $\phi: (-b_0, b_0) \to \mathbb R$ be a positive smooth function such that $\alpha^*(\beta)\in [0,b_0)$. The function $F = \alpha \phi \big( \tfrac{\beta}{\alpha} \big)$ is a Minkowski norm if and only if
\begin{equation*}
\Phi(s, b) := (\phi(s) - s\phi^{\prime}(s)) + (b^2 - s^2)\phi^{\prime\prime}(s) > 0 ,
\end{equation*}
for all $s$ and $b$ with $\vert s \vert \leq b < b_0$. In particular, if $F$ is an $(\alpha, \beta)$-Minkowski norm, then
\begin{equation}
\label{Lambda}
\lambda(s) := \phi(s) - s\phi^{\prime}(s) > 0 \, \text{ for all } \, \vert s \vert < b_0 .
\end{equation}
\end{lemma}

On a manifold $M$, let $\alpha$ be a Riemannian norm and $\beta$ a $1$-form. A Finsler metric $F$ on $M$ is said an \textbf{$(\alpha,\beta)$-metric} when there exists a smooth positive function $\phi: (-b_0,b_0) \to \mathbb{R}$ such that $F_p = \alpha_p \, \phi \big( \tfrac{\beta_p}{\alpha_p} \big)$ for all $p \in M$.

\subsection{Isometries}

The isometries of $(\alpha,\beta)$-metrics are, mostly, diffeomorphisms that preserve both, the Riemannian norm $\alpha$ and the $1$-form $\beta$; see \cite{Voicu2023finsler}. At the present, we write in befitting language the results subsequently needed.

\begin{proposition}
\label{AlphaBeta!Phi}
Let $\alpha, \tilde{\alpha}$ be even real-valued functions on the vector space $V$, i.e. $\alpha (-v) = \alpha (v)$ for all $v \in V$, and likewise for $\tilde \alpha$. Let $\beta, \tilde{\beta}$ be linear functionals on $V$, and let $\phi: (-b_0,b_0) \to \mathbb R$ be a positive smooth function that has nowhere zero derivative. Assume that $\tfrac{\beta(v)}{\alpha(v)}, \tfrac{\tilde \beta(v)}{\tilde \alpha(v)} \in (-b_0,b_0)$ for all $v \in V$. If
\begin{equation*}
\alpha \, \phi \bigg(  \frac{\beta}{\alpha} \bigg) = \tilde{\alpha} \, \phi \bigg( \frac{\tilde{\beta}}{\tilde{\alpha}} \bigg) ,
\end{equation*}
then $\alpha = \tilde{\alpha}$ and $\beta = \tilde{\beta}$.
\end{proposition}
\begin{proof}
Fix $v \in V \setminus 0$. By assumption, we have that
\begin{subequations}
\begin{equation}
\label{AlphaBetaTil}
\alpha(v)\,\phi\!\left(\frac{\beta(v)}{\alpha(v)}\right) = \tilde{\alpha}(v)\,\phi \bigg(\frac{\tilde{\beta}(v)}{\tilde{\alpha}(v)}\bigg) ,
\end{equation}
\begin{equation}
\label{AlphaBetaTilNeg}
\alpha(v)\,\phi\!\left(-\frac{\beta(v)}{\alpha(v)}\right) = \tilde{\alpha}(v)\,\phi \bigg( \!-\frac{\tilde{\beta}(v)}{\tilde{\alpha}(v)}\bigg) .
\end{equation}
\end{subequations}
Define the function $h(s) := \tfrac{\phi(s)}{\phi(-s)}$. Dividing \eqref{AlphaBetaTil} by \eqref{AlphaBetaTilNeg} yields
\begin{equation*}
h \! \left( \frac{\beta(v)}{\alpha(v)} \right) = h \bigg( \frac{\tilde{\beta}(v)}{\tilde{\alpha}(v)} \bigg) .
\end{equation*}
Since $\phi$ is positive and has nowhere zero derivative, $h$ is positive and has nowhere zero derivative. Hence,
\begin{equation}
\label{BetaOverAlphaTil}
\frac{\beta(v)}{\alpha(v)}=\frac{\tilde{\beta}(v)}{\tilde{\alpha}(v)}.
\end{equation}
Replacing \eqref{BetaOverAlphaTil} into \eqref{AlphaBetaTil}, we conclude that $\alpha(v) =\tilde{\alpha}(v)$. Now, replacing this equality into \eqref{BetaOverAlphaTil}, we infer  that $\beta(v)=\tilde{\beta}(v)$.
\end{proof}

\begin{corollary}
\label{AlphaBetaIsometry}
Let $\phi: (-b_0,b_0) \to \mathbb R$ be a positive smooth function. Let $\alpha$ be a Euclidean norm, and $\beta$ a linear functional on $V$ such that $\alpha^*(\beta)\in [0,b_0)$. Consider the function $F := \alpha\phi \big(\tfrac{\beta}{\alpha} \big),$ and the groups
\begin{equation*}
\mathrm{Iso} (F) := \{ \sigma \in \mathrm{Diff} (M) : \sigma^{\ast} F = F \} , \text{ and } \; \mathrm{Iso}_{\beta} (\alpha) := \{ \sigma \in \mathrm{Iso} (\alpha) : \sigma^{\ast} \beta = \beta \} .
\end{equation*}
Then $\mathrm{Iso}_{\beta}(\alpha) \subset \mathrm{Iso}(F)$. If, additionally, the derivative $\phi^{\prime}$ is nowhere vanishing, then $\mathrm{Iso}(F)=\mathrm{Iso}_{\beta}(\alpha)$.
\end{corollary}
\begin{proof}
If $\sigma \in \mathrm{Iso}_{\beta}(\alpha)$, it follows trivially  that $\sigma^{\ast} (F) = F$. Hence, $\mathrm{Iso}_{\beta} (\alpha) \subset \mathrm{Iso} (F)$. Suppose $\phi^{\prime}$ is nowhere vanishing. If $\sigma \in \mathrm{Iso}(F)$, then
\begin{equation*}
F = ( \sigma^{\ast} \alpha ) \phi \!\left( \frac{\sigma^{\ast} \beta}{\sigma^{\ast} \alpha} \right) .
\end{equation*}
It follows from Proposition~\ref{AlphaBeta!Phi} that $\sigma^{\ast} \alpha = \alpha$, and $\sigma^{\ast} \beta = \beta$. Thus, $\sigma \in \mathrm{Iso}_\beta (\alpha)$.
\end{proof}

\begin{lemma}
\label{BetaOverAlphaCte}
Let $F=\alpha \phi (\tfrac{\beta}{\alpha})$ be an $(\alpha,\beta)$-norm on the vector space $V$, and $\bb$ the $\alpha$-dual vector to $\beta$. If $\vt$ is a subspace of $V$ that contains $\bb$, then the ratio $\frac{\beta}{\alpha}$ is constant over the $F$-orthogonal cone to $\vt$, i.e. $\frac{\beta(v)}{\alpha(v)} = c \in \mathbb R$ for all $v \in \nu(\vt)$.
\end{lemma}
\begin{proof}
Once $\frac{\beta(v)}{\alpha(v)}$ is positively homogeneous of degree zero, it suffices to show that the quotient is constant on $\nu^1 (\vt).$

Consider
\begin{equation*}
\mathrm{K} = \{ Q \in \mathbb{O}(\alpha) :  Q \vert_{V} = Id\vert_{\vt} \} .
\end{equation*}
In particular, $Q(\bb) = \bb$, because $\bb \in \vt$, and $Q(\hz) = \hz$ for all $Q \in \mathrm{K}$, where $\hz$ is the subspace $\alpha$-orthogonal to $\vt$. Moreover, by Corollary~\ref{AlphaBetaIsometry}, $\mathrm{K} \subset Iso (F)$.

For each $v \in V$ and $Q \in \mathrm{K}$,
\begin{equation*}
\beta(Q(v)) = \ga (Q(v),\bb) = \ga (Q(v),Q(\bb)) = \ga (v,\bb) = \beta(v) ,
\end{equation*}
where $\ga$ is the Riemannian metric associated with $\alpha$. So $\frac{\beta(Q(v))}{\alpha(Q(v))} = \frac{\beta(v)}{\alpha(v)}$, and it is enough to prove that $\nu^1(\vt) = \mathrm{K}(v)$ for the, hereafter, fixed vector $v\in \nu^{1}(\vt)$.

If $Q \in \mathrm{K}$ and $w \in \vt$, then
\begin{equation*}
\g^{F}_{Q(v)}(Q(v),w) = \g^{F}_{Q(v)}(Q(v),Q(w)) = \g^{F}_v (v,w) = 0 ,
\end{equation*}
because $Q\vert_{\vt} = Id\vert_{\vt}$, and $Q \in Iso(F)$. Thus, the orbit $\mathrm{K}(v)$ is a submanifold of $\nu^1(\vt)$. To conclude the prove, observe that $\mathrm{K}$ can be identified with the group of linear isometries for $(\hz, \alpha\vert_{\hz})$. In particular, 
$\mathrm{K}(v)$ is closed. But, since $\dim \mathrm{K}(v) = \dim \hz -1 = \dim \nu^1(\vt)$, the orbit $\mathrm{K}(v)$ must be open in $\nu^1(\vt)$. This completes the desired equality for $\dim (\hz) \geq 2$, because $\nu^1(\vt)$ is connected. At last, when $\hz$ is $1$-dimensional, $\mathrm{K}$ contains only the identity and a reflection $R$ across $\vt$. In this case, $\nu(\vt)$ is also $1$-dimensional, so $v$ and $R(v)$ span its connected components. Hence, $\nu^1 (\vt)$ consists precisely of $v$ and $R(v)$.
\end{proof}

\subsection{Legendre transformation}

The \textbf{Legendre transform} (or \textbf{Legendre map}) of a  Finsler metric $F$ is the map $\ell^F: TM \to T^*M$ which associates to each $v \in TM$ the linear functional $\ell^F_v \in T_{\tau(v)}^*M$ defined as
\begin{equation*}
\ell^F_v(w) := \left. \frac{d}{dt} \left[ \frac{1}{2} F^2(v + tw) \right] \right\vert_{t=0}.
\end{equation*}
By Lemma~\ref{FundTensor}, $\ell^F_v(w)=\g^F_v(v, w)$ for all $v \in TM \setminus 0$. So $\ell^F$ restricts to a diffeomorphism from $TM \setminus 0$ onto $T^*M \setminus 0$.

\begin{lemma}
\label{AlphaBetaLegendre} 
If $F = \alpha \phi \big( \tfrac{\beta}{\alpha} \big)$ is an $(\alpha,\beta)$-metric, then its Legendre map is given by
\begin{equation}
\label{AlphaBetaLegendreEq}
\begin{aligned}
\ell^F_v(w) &= F(v) \ga\!\left( \lambda \!\left( \tfrac{\beta(v)}{\alpha(v)} \right) \frac{v}{\alpha(v)} + \phi^{\prime} \!\left( \tfrac{\beta(v)}{\alpha(v)} \right) \bb , w \right) \\
&= \ga \left( \rho \left( \tfrac{\beta(v)}{\alpha(v)} \right) v + F(v) \phi^{\prime} \!\left( \tfrac{\beta(v)}{\alpha(v)} \right) \bb , w \right)
\end{aligned}
\end{equation}
where $\ga$ is the Riemannian metric associated with $\alpha$, $\bb$ is the vector field $\alpha$-dual to $\beta$, $\rho(s) = \phi(s) \lambda(s)$, and $\lambda$ as in \eqref{Lambda}.
\end{lemma}
\begin{proof}
The expression for the fundamental tensor of an $(\alpha,\beta)$-metric is well-known; e.g. \cite[Equation~1.5]{Shibata1984invariant}, \cite[Equation~2.2]{Matsumoto1992theory}, or \cite[p.~5]{Chern2005riemann}. In the present notation, fixed $v \in TM \setminus 0$,
\begin{equation*}
\begin{aligned}
\g^F_v (u,w) = & \rho \!\left( \tfrac{\beta (v)}{\alpha (v)} \right) \ga (u,w) + \rho_0 \!\left( \tfrac{\beta (v)}{\alpha (v)} \right) \beta (u) \beta (w) \\
& + \rho_1 \!\left( \tfrac{\beta (v)}{\alpha (v)} \right) \left[ \beta (u) \frac{\ga (v,w)}{\alpha (v)} + \beta (w) \frac{\ga (v,u)}{\alpha (v)} \right] \\
& - \tfrac{\beta (v)}{\alpha (v)} \rho_1 \!\left( \tfrac{\beta (v)}{\alpha (v)} \right) \frac{\ga (v,u)}{\alpha (v)} \frac{\ga (v,w)}{\alpha (v)} ,
\end{aligned}
\end{equation*}
for all $u,w \in T_{\tau(v)} M$, where $\rho (s) = \phi(s) (\phi(s) - s \phi^{\prime}(s))$, $\rho_0 (s) = \phi(s) \phi^{\prime\prime}(s) + (\phi^{\prime}(s))^2$, and $\rho_1(s) = \phi(s) \phi^{\prime}(s) - s \rho_0(s)$. By Lemma~\ref{FundTensor}, making $u=v$ above gives
\begin{equation*}
\begin{aligned}
\ell^F_v (w) &= \rho \!\left( \tfrac{\beta (v)}{\alpha (v)} \right) \ga (v,w) + \rho_0 \!\left( \tfrac{\beta (v)}{\alpha (v)} \right) \beta (v) \beta (w) + \rho_1 \!\left( \tfrac{\beta (v)}{\alpha (v)} \right) \beta (w) \alpha (v) \\
&= \rho \!\left( \tfrac{\beta (v)}{\alpha (v)} \right) \ga (v,w) + \phi\!\left( \tfrac{\beta (v)}{\alpha (v)} \right) \phi^{\prime}\!\left( \tfrac{\beta (v)}{\alpha (v)} \right) \beta (w) \alpha (v) .
\end{aligned}
\end{equation*}
Taking $\beta(\cdot) = \ga (\bb,\cdot)$ implies
\begin{equation*}
\ell^F_v (w) = \rho \!\left( \tfrac{\beta (v)}{\alpha (v)} \right) \ga (v,w) + F(v) \phi^{\prime}\!\left( \tfrac{\beta (v)}{\alpha (v)} \right) \ga (\bb,w) ,
\end{equation*}
which yields the equality \eqref{AlphaBetaLegendreEq}.

Alternatively, \eqref{AlphaBetaLegendreEq} can be established by straightforward computation as follows. If $v \in TM \setminus 0$, and $w\in T_{\tau(v)}M$, then
\begin{equation*}
\begin{aligned}
\ell^F_v (w) &= \left. \frac{d}{dt} \left[ \tfrac{1}{2} \alpha^2 (v+tw) \phi^2  \!\left( \tfrac{\beta (v+tw)}{\alpha (v+tw)} \right) \right]\right\vert_{t=0} \\
&= \ga \left( v,w \right) \phi^2 \!\left( \tfrac{\beta(v)}{\alpha(v)} \right) + \alpha^2 (v) \phi \!\left( \tfrac{\beta(v)}{\alpha(v)} \right) \phi^{\prime} \!\left( \tfrac{\beta(v)}{\alpha(v)} \right) \left[ \frac{\beta(w)}{\alpha(v)} - \frac{\beta(v)}{\alpha^2(v)}\frac{\ga \left( v,w \right)}{\alpha(v)} \right] \\
&= \phi \!\left( \tfrac{\beta(v)}{\alpha(v)} \right) \left\lbrace \phi \!\left( \tfrac{\beta(v)}{\alpha(v)} \right) \ga \left( v,w \right) + \phi^{\prime} \!\left( \tfrac{\beta(v)}{\alpha(v)} \right) \left[ \alpha(v) \beta(w) - \frac{\beta(v)}{\alpha(v)} \ga \left( v,w \right) \right] \right\rbrace \\ 
&= \phi \!\left( \tfrac{\beta(v)}{\alpha(v)} \right) \ga \left( \left\lbrace \phi \!\left( \tfrac{\beta(v)}{\alpha(v)} \right) - \phi^{\prime} \!\left( \tfrac{\beta(v)}{\alpha(v)} \right) \frac{\beta(v)}{\alpha(v)} \right\rbrace v + \phi^{\prime} \!\left( \tfrac{\beta(v)}{\alpha(v)} \right) \alpha(v) \bb , w \right) \\
&= F(v) \ga \left( \left\lbrace \phi \!\left( \tfrac{\beta(v)}{\alpha(v)} \right) - \tfrac{\beta(v)}{\alpha(v)} \phi^{\prime} \!\left( \tfrac{\beta(v)}{\alpha(v)} \right) \right\rbrace \frac{v}{\alpha(v)} + \phi^{\prime} \!\left( \tfrac{\beta(v)}{\alpha(v)} \right) \bb , w \right) .
\end{aligned}
\end{equation*}
\end{proof}

Lastly, consider Lemma~\ref{FundTensor} while making $w=v$ in equation~\eqref{AlphaBetaLegendreEq} to conclude an $(\alpha,\beta)$-metric can be expressed as
\begin{equation}
\label{AlphaBetaRep}
F(v) = \ga\!\left( \lambda \!\left( \tfrac{\beta(v)}{\alpha(v)} \right) \frac{v}{\alpha(v)} + \phi^{\prime} \!\left( \tfrac{\beta(v)}{\alpha(v)} \right) \bb , v \right) ,
\end{equation}
for $\ga$, $\lambda$, and $\bb$ defined as in Lemma~\ref{AlphaBetaLegendre}.
\section{Finsler submersions}

A submersion $\pi: (M_1,F_1) \to (M_2,F_2)$ between two Finsler manifolds is called a \textbf{Finsler submersion} if, at every point $p \in M_1$, $d\pi_p (B^{F_1}_1 (0)) = B^{F_2}_1 (0)$, i.e. $d\pi_p$ sends the unit ball of the Minkowski space $(T_p M_1, (F_1)_p)$ to the unit ball of the Minkowski space $(T_{\pi(p)} M_2, (F_2)_{\pi(p)})$.

\begin{lemma}\cite[Proposition~2.1]{Alvarez2001isometric}
If $\pi: (M_1,F_1) \to (M_2,F_2)$ is a Finsler submersion, then, for each $p \in M_1$ and $w \in T_{\pi(p)} M_2$,
\begin{equation*}
F_2 (w) = \inf \{ F_1 (v) : v \in T_p M_1 \, , \; d\pi_p (v) = w\} .
\end{equation*}
\end{lemma}

In particular, $F_2 (d\pi_p (v)) \leq F_1 (v)$ for all $v \in T_p M_1$. The vectors for which equality holds are said to be \textbf{horizontal}. Furthermore, the set of horizontal vectors at $p \in M_1$ coincides with the orthogonal cone $\nu_p (\pi^{-1} (s))$, where $s := \pi(p)$. To simplify notation, we write $\nu_p (\pi^{-1} ( \pi(p) ))$ as $\nu_p (\pi)$ and call it the \textbf{horizontal cone} at $p$. Similarly, the set $\nu(\pi^{-1} (s))$ of all orthogonal vectors to $\pi^{-1} (s)$ is the \textbf{horizontal cone} to said fiber, and denoted $\nu(\pi)_s$. As in the Riemannian case, the Finsler submersion admits a lifting property for geodesics.

\begin{lemma}\cite[Theorem~3.1]{Alvarez2001isometric}
\label{GeodesicLift}
Let $\pi: (M_1,F_1) \to (M_2,F_2)$ be a Finsler submersion. For each unit speed geodesic $\gamma: (-\varepsilon , \varepsilon) \to M_2$ and each point $p \in \pi^{-1}(\gamma(0))$, there is one and only one unit speed geodesic $\tilde{\gamma}: (-\varepsilon , \varepsilon) \to M_1$ such that $\pi \circ \tilde{\gamma} = \gamma$ and  $\tilde \gamma^{\prime} (0) \in \nu_p (\pi)$.
 \end{lemma}

It can be verified that a submersion $\pi: (M_1,F_1) \to (M_2,F_2)$ is a Finsler submersion if and only if, for each $p \in M_1$,
\begin{equation}
\label{FinslerSubmHorizVec}
\g^{F_1}_v (v,w) = \g^{F_2}_{d\pi_p (v)} (d\pi_p (v) , d\pi_p (w)) ,
\end{equation}
for all $v \in \nu_p (\pi)$, and all $w \in T_p M_1$.

\begin{proposition}\cite[Proposition~2.2]{Alvarez2001isometric}
\label{MinkowskiRiemannSubm}
Let $\pi : (V_1,F_1) \to (V_2,F_2)$ be a linear submersion of Minkowski spaces. If $v$ is a horizontal vector, i.e. $v \in \nu_v (\pi)$, then $\pi: (V_1,\g^{F_1}_v) \to (V_2,\g^{F_2}_{\pi(v)})$ is an isometric submersion of Euclidean spaces. 
\end{proposition}

The next result allows us to ``flatten'' the horizontal cones of linear submersions between Minkowski spaces.

\begin{lemma}\cite[Lemma~2.9]{Alexandrino2019singular}
\label{FlattingHorizCone}
Let $\pi: V_1 \to V_2$ be a linear submersion between vector spaces. If $F_1$ is a Minkowski norm on $V_1$, and $\nu (\pi)_0$ is the set of $F_1$-orthogonal vectors to $\pi^{-1}(0)$, then:
\begin{enumerate}[label=\textup{(\roman*)}]
\item the map $\psi := \pi\vert_{\nu(\pi)_0}:\nu(\pi)_0 \to V_{2}\setminus 0$ is a positively $1$-homogeneous diffeomorphism that can be continuously extended to zero; and

\item the function $F_{2}(v) := F_{1}(\psi^{-1}(v))$ is the unique Minkowski norm on $V_2$ that makes $\pi : (V_1,F_1) \to (V_2,F_2)$ a linear Finsler submersion.
\end{enumerate}
\end{lemma}
\section{\texorpdfstring{$(\alpha,\beta)$}{alpha,beta}-Finsler submersions}

\begin{proposition}
\label{AlphaBetaSubm}
Let $\pi: (M_1,F_1) \to (M_2,F_2)$ be a Finsler submersion. Suppose $F_1$ is an $(\alpha,\beta)$-metric; namely, $F_1 = \alpha_1 \phi \big( \tfrac{\beta_1}{\alpha_1} \big)$. Denote the Riemannian metric associated to $\alpha_1$ as $\ga_1$, and the $\alpha_1$-dual vector field to $\beta_1$ simply as $\bb$.
\begin{enumerate}[label=\textup{(\roman*)}]
\item\label{s1} If $\bb \neq 0$, then $(T_{\pi(p)} M_2 , (F_2)_{\pi(p)})$ is isometric to an $(\alpha,\beta)$-Minkowski space for each $p \in M_1$.

\item\label{s2} If $\bb$ is $\alpha_1$-orthogonal to the fibers, and $\phi$ has nowhere zero derivative, then
\begin{itemize}
\item $\bb$ is foliated and horizontal;

\item $F_2$ is an $(\alpha,\beta)$-metric (for the function $\phi$), i.e. $F_2 = \alpha_2 \phi \big( \tfrac{\beta_2}{\alpha_2} \big)$ for a Riemannian norm $\alpha_2$ and a $1$-form $\beta_2$ on $M_2$; and

\item $\pi: (M_1,\ga_1) \to (M_2,\ga_2)$ is a Riemannian submersion. Here, as before, $\ga_2$ denotes the Riemannian metric associated to $\alpha_2$.
\end{itemize}

\item\label{s3} If $\bb$ is tangent to the fibers of $\pi$, then there exist a function $\kappa: M_1 \to \mathbb{R}$, and a Riemannian metric $\ga_2$ on $M_2$ such that $\pi: (M_1,\kappa\ga_1) \to (M_2,\ga_2)$ is a Riemannian submersion. In addition, $F_2$ is precisely $\alpha_2$.
\end{enumerate}
\end{proposition}

\begin{remark}
\label{LocalMinkowski}
Item \ref{s1} of Proposition~\ref{AlphaBetaSubm} is requiring the Finsler metric $F_2$ to be an $(\alpha,\beta)$-Minkowski norm on each tangent space. It is important to observe that this is not equivalent to $F_2$ being an $(\alpha,\beta)$-metric on $M_2$, as the real function ``$\phi$'' that would generate the Finsler metric pointwise may vary with the points $p \in M_1$ over the same fiber of $\pi$.
\end{remark}

\begin{proof}
Let $\f$ denote the partition of $M_1$ into the fibers of $\pi$, i.e. $\f = \{ \pi^{-1}(s) \}_{s \in M_2}$. Each connected component of a fiber will be called a ``\emph{leaf}'' of $\f$. As per usual, $\hz$ denotes the $\alpha_1$-orthogonal distribution to the leaves of $\f$, and $\vt$ the vertical space, i.e. the tangent space to the leaves of $\f$.

To prove item~\ref{s1}, our goal is to find a candidate that can represent the pullback of $F_2$ on the subspace $\hz_p$ for each $p \in M_1$; vide \eqref{F2candidate}. These candidates will be defined on a tubular neighborhood $\tub_{\delta}(L_q)$ of the leaf $L_q$ through $q \in M_1$. The radius $\delta > 0$ should be small enough so that there is an $\mathcal{F}$-foliated tangent frame $\{ \e_i \}$ for $\hz$ over the tubular neighborhood.

For each $p \in \tub_{\delta}(L_q)$ and $v \in T_p M_1$, let $v^h$ denote the $\alpha_1$-orthogonal projection of $v$ onto $\hz$. Let $\varphi_p: \hz_p \to \nu_p (\pi)$ be the ``\emph{conical lifting}", i.e. the inverse of the $\alpha_1$-orthogonal projection $\psi_p: v \in \nu_p (\pi) \mapsto v^h \in \hz_p$; recall Lemma~\ref{FlattingHorizCone}. Consider the vector field $\bb^h \in \mathfrak{X} (\hz)$, and the group $\mathbb{O}_{\bb^h (p)}$ of linear isometries on $(\hz_p, (\alpha_1)_p)$ which fix the vector $\bb^{h}(p)$. The linear action of $\mathbb{O}_{\bb^h (p)}$ extends to $T_p M_1$ acting as the identity on $\vt_p$. Regarding $\mathbb{O}_{\bb^h (p)}$ as a subgroup of isometries on $T_p M_1$, it fixes the horizontal cone $\nu_p (\pi)$; consequence of Lemma~\ref{AlphaBetaLegendre} keeping in mind that $\bb(p)$ is invariant under $\mathbb{O}_{\bb^h (p)}$. So
\begin{equation}
\label{ConicLiftInvariance}
\varphi_p (Q y) = Q \varphi_p (y) \, \text{ for all } \; Q \in \mathbb{O}_{\bb^h (p)} \, \text{ and } \, y \in \hz_p .
\end{equation}

Define the following positively zero-homogeneous map $\vartheta_p$ on $\hz_p$:
\begin{equation}
\label{FunctCandidate}
\vartheta_p (y) := \lambda \!\left( \tfrac{\beta_1 \circ \varphi_p (y)}{\alpha_1 \circ \varphi_p (y)} \right) \frac{\tilde{\alpha}_2 (y)}{\alpha_1 \circ \varphi_p (y)} + \phi^{\prime} \!\left( \tfrac{\beta_1 \circ \varphi_p (y)}{\alpha_1 \circ \varphi_p (y)} \right) \frac{\tilde{\beta}_2 (y)}{\tilde{\alpha}_2 (y)} ,
\end{equation}
where
\begin{equation*}
\tilde{\alpha}_2 := \alpha_1 \vert_{\hz} \, \text{ and } \, \tilde{\beta}_2 (\cdot) := \ga_1 (\bb^{h},\cdot) \vert_{\hz} .
\end{equation*}
Equations~\eqref{ConicLiftInvariance} and~\eqref{FunctCandidate} imply that
\begin{equation*}
\vartheta_p (Qy) = \vartheta_p (y) \, \text{ for all } \, Q \in \mathbb{O}_{\bb^{h}(p)} \, \text{ and } \, y \in \hz_p .
\end{equation*}
Then, by Proposition~\ref{AlphaBetaSymmetry}, there exists a positive smooth real function $\tilde\phi_p$, depending on $p \in \tub_{\delta} (L_q)$, such that
\begin{equation}
\label{F2candidate}
\tilde{\alpha}_2 (y) \vartheta_p (y) = \tilde{\alpha}_2 (y) \tilde{\phi}_p \! \left( \frac{\tilde{\beta}_2 (y)}{\tilde{\alpha}_2 (y)} \right) \, \text{ for all } \; y \in \hz_p .
\end{equation}
Let us briefly comment on its construction. Given a foliated $\alpha_1$-orthonormal local frame $\{ \e_i \}$ for the distribution $\hz$ with $\bb^{h}(p) = \alpha_1 (\bb^{h}(p)) \e_n (p)$, the function $\tilde\phi_p$ is defined as
\begin{equation*}
\tilde{\phi}_p (\alpha_1 (\bb^{h}(p)) s) = \vartheta_p (0, \cdots, 0, \sqrt{1-s^2}, s) .
\end{equation*}

Equation~\eqref{F2candidate} gives a representative for $(F_2)_{\pi(p)}$ in the subspace $\hz_p$. It remains to verify this is the right one, i.e.
\begin{equation*}
F_{2} (d\pi_{p} (y)) = \tilde{\alpha}_{2} (y) \tilde{\phi}_p \!\left( \frac{\tilde{\beta}_{2}(y)}{\tilde{\alpha}_{2}(y)} \right) \, \text{ for all } \; y \in \hz_p .
\end{equation*}

From Lemma~\ref{AlphaBetaLegendre},
\begin{equation}
\label{HorizConeCharacterization}
v \in \nu_p (\pi) \, \text{ if and only if } \, \lambda \left( \tfrac{\beta_1 (v)}{\alpha_1 (v)} \right) \frac{v}{\alpha_1 (v)} + \phi^{\prime} \!\left( \tfrac{\beta_1 (v)}{\alpha_1 (v)} \right) \bb(p) \in \hz_p  
\end{equation}
Now, we combine several equations discussed earlier. If $v \in \nu_p (\pi)$, then
\begin{equation*}
\begin{aligned}
F_1(v) &\overset{\eqref{AlphaBetaRep}}{=} \ga_1 \!\left( \lambda \!\left( \tfrac{\beta_1 (v)}{\alpha_1 (v)} \right) \frac{v}{\alpha_1 (v)} + \phi^{\prime} \!\left( \tfrac{\beta_1 (v)}{\alpha_1 (v)} \right) \bb(p) , v \right) \\
&\overset{\eqref{HorizConeCharacterization}}{=} \ga_1 \!\left( \lambda \!\left( \tfrac{\beta_1 (v)}{\alpha_1 (v)} \right) \frac{v}{\alpha_1 (v)} + \phi^{\prime} \!\left( \tfrac{\beta_1 (v)}{\alpha_1 (v)} \right) \bb(p) , v^h \right) \\
&= \lambda \!\left( \tfrac{\beta_1 (v)}{\alpha_1 (v)} \right) \frac{\tilde\alpha_2^2 (v^h)}{\alpha_1 (v)} + \phi^{\prime} \!\left( \tfrac{\beta_1 (v)}{\alpha_1 (v)} \right) \tilde\beta_2(v^h) \\
&= \tilde\alpha_2 (v^h) \left\lbrace \lambda \!\left( \tfrac{\beta_1\circ\varphi_p (v^h)}{\alpha_1\circ\varphi_p (v^h)} \right) \frac{\tilde\alpha_2 (v^h)}{\alpha_1\circ\varphi_p (v^h)} + \phi^{\prime} \!\left( \tfrac{\beta_1\circ\varphi_p (v^h)}{\alpha_1\circ\varphi_p (v^h)} \right) \frac{\tilde\beta_2(v^h)}{\tilde\alpha_2 (v^h)} \right\rbrace \\
&\overset{\eqref{FunctCandidate}}{=} \tilde\alpha_2 (v^h) \vartheta_p (v^h) \\
&\overset{\eqref{F2candidate}}{=} \tilde\alpha_2 (v^h) \tilde{\phi}_p \!\left( \frac{\tilde{\beta}_{2} (v^h)}{\tilde{\alpha}_{2} (v^h)} \right) .
\end{aligned}
\end{equation*}
Summarizing,
\begin{equation}
\label{F1HorizCharact}
F_1 (v) = \tilde{\alpha}_2 (v^h) \tilde{\phi}_p \!\left( \frac{\tilde{\beta}_2 (v^h)}{\tilde{\alpha}_2(v^h)} \right) \, \text{ for all } v \in \nu_p (\pi) .
\end{equation}
Once $\pi: (M_1,F_1) \to (M_2,F_2)$ is a Finsler submersion, for each $v \in \nu_p (\pi)$ we have
\begin{equation}
\label{F2Charact}
F_2 (d\pi_p (v^h)) = F_2 (d\pi_p (v)) = F_1 (v) \overset{\eqref{F1HorizCharact}}{=} \tilde{\alpha}_{2}(v^h) \tilde{\phi}_p \!\left( \frac{\tilde{\beta}_{2}(v^{h})}{\tilde{\alpha}_{2}(v^{h})} \right) ,
\end{equation}
whence item~\ref{s1}. As pointed in Remark~\ref{LocalMinkowski}, we stress that $\tilde{\phi}_p$ generally depends on $p \in \tub_{\delta}(L_q) \subset M_1$. Even for $p \in L_q$, the function $\tilde{\phi}_q$ need not to be equal to $\tilde{\phi}_p$, so they do not induce a function $\tilde \phi$ on a neighborhood of $\pi(q)$.

For the proof of item~\ref{s2}, suppose $\bb$ is $\alpha_1$-orthogonal to the leaves of $\mathcal{F}$, i.e. $\bb \in \mathfrak{X}(\hz)$. Under this assumption,
\begin{equation*}
\bb^{h} = \bb \, \text{ and } \, \tilde{\beta}_2 = \beta_1\vert_{\hz} .
\end{equation*}
Replacing these two equalities into \eqref{FunctCandidate}, we infer that
\begin{equation}
\label{EqualityTilPhipPhi}
\tilde{\phi}_p = \phi .
\end{equation}
Moreover, $\nu_p (\pi) = \hz_p$ for all $p \in M_1$ by \eqref{HorizConeCharacterization}. In such case, for any $v \in \nu_p (\pi)$ and $w \in \nu_q (\pi)$ with $d\pi_p (v) = d\pi_q (w)$, we obtain
\begin{equation*}
\begin{aligned}
\tilde\alpha_2 (v) \phi \!\left( \frac{\tilde\beta_2 (v)}{\tilde\alpha_2 (v)} \right) &\overset{\eqref{F1HorizCharact}}{=} F_1(v) = F_2(d\pi_p (v)) \\
&= F_2(d\pi_q (w)) = F_1(w) \overset{\eqref{F1HorizCharact}}{=} \tilde\alpha_2 (w) \phi \!\left( \frac{\tilde\beta_2 (w)}{\tilde\alpha_2 (w)} \right) .
\end{aligned}
\end{equation*}
When $\phi^{\prime}$ is nowhere vanishing, Proposition~\ref{AlphaBeta!Phi} applies to ensure
\begin{equation}
\label{EqualityAlphaBeta}
\tilde{\alpha}_{2}(v) = \tilde{\alpha}_{2}(w) \, \text{ and } \, \tilde{\beta}_{2}(v) = \tilde{\beta}_{2}(w) .
\end{equation}
Putting together equations~\eqref{EqualityTilPhipPhi} and~\eqref{EqualityAlphaBeta}, we conclude item~\ref{s2}. 

Finally, for item~\ref{s3}, assume $\bb$ is tangent to the leaves of $\mathcal{F}$, i.e. $\bb \in \mathfrak{X}(\vt)$. As seen in Lemma~\ref{BetaOverAlphaCte}, the function $h: M_1 \to \mathbb{R}$, given by
\begin{equation}
\label{BetaOverAlphaFunct}
h(p) := \frac{\beta_1 (v)}{\alpha_1 (v)} \; \text{ for any } v \in \nu_p (\pi) ,
\end{equation}
is well-defined. For each $q \in M_1$, consider a smooth (local) vector field $X$ such that $X(p) \in \nu_p^{1}(\pi)$ on some neighborhood of $q$. Since $h(p) = \tfrac{\beta_1 (X(p))}{\alpha_1 (X(p))}$, the function $h$ is smooth around $q$. Hence, $h$ is smooth.

Reconsider \eqref{FunctCandidate} with the fact that $\tilde{\beta}_{2}$ vanishes on $\mathcal{H}$, by hypothesis. Then,
\begin{equation*}
\vartheta_p (y) = \lambda \circ h (p) \, \frac{\tilde\alpha_2 (y)}{\alpha_1 \circ \varphi_p (y)} \, \text{ for all } \, y \in \hz_p .
\end{equation*}
By the argument proving equation~\eqref{F1HorizCharact}, we have
\begin{equation}
\label{F1HorizBVert}
F_1 \circ \varphi_p (y) = \vartheta_p (y) \, \tilde\alpha_2 (y) = \lambda \circ h (p) \, \frac{\tilde\alpha_2^2 (y)}{\alpha_1 \circ \varphi_p (y)} \, \text{ for all } \, y \in \hz_p .
\end{equation}
Using $\rho = \phi \lambda$ and $F_1 = \alpha_1 \phi \big( \tfrac{\beta_1}{\alpha_1} \big)$, we obtain that
\begin{equation*}
F_2^2 (d\pi_p (y)) = F_2^2 (d\pi_p \circ \varphi_p (y)) = F_1^2 (\varphi_p (y)) \overset{\eqref{F1HorizBVert}}{=} \rho \circ h (p) \, \tilde\alpha_2^2 (y) \, \text{ for all } \, y \in \hz_p .
\end{equation*}
So item~\ref{s3} follows from setting $\kappa := \rho \circ h$.
\end{proof}
\section{Singular Finsler foliations}

\begin{definition}
Let $\f := \{L_p\}_{p \in M}$ be a partition of the manifold $M$ by connected immersed submanifolds without intersections, where $L_p$ is called the \textbf{leaf} of $\f$ through the point $p \in M$. Then $\f$ is a \textbf{singular foliation} if each point $q \in M$ admits a neighborhood $U_q$ such that for any vector $v \in T_q L_q$ there exists a vector field $X \in \mathfrak{X} (U_q)$ satisfying:
\begin{enumerate}[label=\textup{(\roman*)}]
\item $X(q) = v$;

\item $X(p) \in T_p L_p$ for all $p \in U_{q}$.
\end{enumerate}
The leaves of $\f$ are allowed to have different dimensions. Each union of leaves with the same dimension is called a \textbf{stratum}, and the union of leaves with maximal dimension is the \textbf{regular stratum}. When all the leaves have  the same dimension, $\f$ is called a \textbf{regular foliation} or just a foliation. 
\end{definition}

\begin{remark}
\label{TrivNbhd}
The above definition of singular foliation is equivalent to the existence around each point $q \in M$ of a submersion $\pi: U_q \to S_q$, from a possibly smaller neighborhood $U_q$ onto an embedded submanifold $S_q$ transverse to $L_q$ (called \textbf{slice}), such that each fiber $\pi^{-1}(s)$ is contained in $L_s$ and $P_q := \pi^{-1}(q)$ is an open, relatively compact, and connected subset of $L_q$ which contains the point $q$ (called \textbf{plaque}). When $\f$ is a (regular) foliation, such neighborhood $U_q$ is commonly called a \textbf{trivializing neighborhood}, and we shall adopt the name in general. Given a singular foliation $\f$ and fixed a trivializing neighborhood $U_q$, we set the convention of writing $\f^{\pi}$ for the (regular) foliation $\{\pi^{-1} (s)\}_{s \in S_q}$; in particular, $L_q^{\pi} = P_q$. Finally, $\f$ is said \textbf{locally closed} if $\f \cap S_q$ is always compact.
\end{remark}

A good example of singular foliation to keep in mind is the partition of a manifold $M$ into the orbits of a Lie group action. More generally, one can think of the partition of $M$ into the orbits of a geometric control. For the present, we are interested on singular foliations with an additional metric property.

\begin{definition}
A singular foliation $\f = \{L_p\}_{p \in M}$ on a (complete) Finsler manifold $(M,F)$ is called a \textbf{singular Finsler foliation} (or SFF, for short) if each geodesic $\gamma: \mathbb{R} \to M$ meeting the leaf $L_{\gamma(0)}$ orthogonally is \textbf{horizontal}, i.e. orthogonal to each leaf it intersects (whether they will meet or already have met).
 \end{definition}

\begin{remark}
When a SFF has all leaves of the same dimension (i.e. it is a regular foliation), we call it a \textbf{(regular) Finsler foliation}. Due to the lifting property for geodesics (recall Lemma~\ref{GeodesicLift}), the partition of $M_1$ into the connected components of the fibers of a Finsler submersion $\pi: (M_1,F_1) \to (M_2,F_2)$ is a (regular) Finsler foliation. Conversely, any (regular) Finsler foliation is locally described as a Finsler submersion.
\end{remark}

For some examples of SFFs, just consider the partition of a Finsler manifold into the orbits of an isometric Finsler action. In alternative, one can construct infinitely many non-homogeneous examples in Randers spaces; \cite[Example 2.13]{Alexandrino2019singular}. Briefly, start with a non-homogeneous singular Riemannian foliation (SRF) on a Riemannian manifold $(M,\alpha)$ (which exist by \cite{Radeschi2014clifford}) and choose $\bb$ ($\alpha$-dual to $\beta$) as a foliated and horizontal vector field. The initial foliation is a SFF with respect to the Randers metric $R = \alpha + \beta$. We can also generalize this sort of construction to $(\alpha,\beta)$-spaces; see Example~\ref{SphereSFF}.

Throughout this section, assume \emph{$\f$ is a locally closed singular Finsler foliation (SFF) on a complete Finsler space $(M,F)$}.

Consider a plaque $P_q$ as in Remark~\ref{TrivNbhd}. Let $f_q^{+}: U_q \to \mathbb{R}$ be the \textbf{future distance} with respect to $P_q$, i.e. $f_q^{+} (x) := d(P_q,x)$. Similarly, let $f_q^{-}: U_q \to \mathbb{R}$ be the \textbf{past distance} with respect to $P_q$, i.e. $f_q^{-} (x) := d(x,P_q)$. Reducing the neighborhood $U_q$ if necessary, we assume $f_q^{+}$ and $f_q^{-}$ are smooth on $U_q \setminus P_q$. Define the \textbf{future tubular neighborhood} of radius $\delta > 0$ around $P_q$ as
\begin{equation*}
\tub_{\delta}^{+} (P_q) := (f_q^{+})^{-1} [0,\delta) ,
\end{equation*}
and the \textbf{future cylinder} of radius $\delta > 0$ as
\begin{equation*}
C_{\delta}^{+} (P_q) := (f_q^{+})^{-1}(\delta) .
\end{equation*}
The definitions of \textbf{past tubular neighborhood and cylinder} are analogous.

An important property of SFFs is that each leaf of $\f \cap U_q$ is \textbf{locally forward and backward equidistant}, i.e. simultaneously contained in past and future cylinders. Precisely,
\begin{equation*}
\text{if } \, x \in U_q \, \text{ then } \, (L_x \cap U_q) \subset C_{\delta_1}^{+} (P_q) \cap C_{\delta_2}^{-}(P_q) .
\end{equation*}
It is worth noting the above condition provides an equivalent definition for a SFF; see \cite[Lemma~3.7]{Alexandrino2019singular}.

Given a future tubular neighborhood $\tub_{\delta}^{+} (P_q)$ of the plaque $P_q$, there is a natural projection $\rho^{+}:\tub_{\delta}^{+}(P_q)\to P_{q} $ defined as follows: $\rho^{+}(x)$ is the only point in $P_q$ such that the unit speed geodesic $\gamma^{+}$ from $\rho^{+}(x)$ to $x$ is orthogonal to $P_q$. The map $\rho^{+}$ is called \textbf{future metric projection}. Likewise we define the \textbf{past metric projection}.

\begin{lemma}\cite[Lemma~3.10]{Alexandrino2019singular}
For each $x\in \tub_{\delta}^{+} (P_q)$, the restriction of the future metric projection to $P_x := L_x\cap \tub_{\delta}^{+}(P_q)$ is a surjective submersion onto $P_q$. The same holds true for the past metric projection.
\end{lemma}

For a future tube $\tub_{\delta}^{+} (P_q)$ still fixed and $\lambda \in (0,1]$, we define the \textbf{future homothetic transformation} $h_{\lambda}^{+}: \tub_{\delta}^{+}(P_q) \to \tub_{\lambda \delta}^{+}(P_q)$ as  $h_{\lambda}^{+} \big(\gamma^{+} (r)\big) := \gamma^{+}(\lambda r)$, where $\gamma^{+}$ is a unit speed geodesic starting orthogonal to $P_q$ and $r>0$. A similar definition applies to the \textbf{past homothetic transformation} $h^{-}_{\lambda}$. These transformations are related to SFFs through the following result, known as the \emph{Homothetic Transformation Lemma}.

\begin{lemma}\cite[Lemma~3.14]{Alexandrino2019singular}
\label{HomotheticLemma}
The future (resp. past) homothetic transformation
$h^{+}_{\lambda} $ (resp. $h^{-}_{\lambda} $) sends plaques to plaques of a singular Finsler foliation.
\end{lemma}

Next, we can ``flatten" the horizontal cones along a slice $S_q$, so that the Finsler metric in the leaf space of the foliation $\f\cap S_q$ coincides with the Finsler metric in the leaf space of the foliation $\f \cap U_q$. Furthermore, the foliation becomes diffeomorphic to a SFF on a Minkowski space.

\begin{theorem}\cite[Theorem~1.3]{Alexandrino2019singular}
\label{SliceThm}
Given a point $q \in M$, consider the trivializing neighborhood $U_q$ and the slice $S_q$ presented in Remark~\ref{TrivNbhd}. There exists a Finsler metric $\widehat{F}$ on $S_q$ such that:
\begin{enumerate}[label=\textup{(\roman*)}]
\item the foliation $\f_q := \f\cap S_q$ is a SFF with respect to $\widehat{F}$;

\item the distance between the leaves of $\f_q$ (with respect to $\widehat{F}$) and the distance between the leaves of $\f\cap U_q$ (with respect to $F$) coincide;

\item the slice foliation  $\f_q$ on $S_q$ is foliated diffeomorphic to a singular Finsler foliation on an open subset of the Minkowski space  $(T_q S_q, \widehat{F}_{q})$. 
\end{enumerate}
\end{theorem}

We refer to the above result as the \emph{Slice Theorem}, since it extends the classic slice theorem for isometric actions (see \cite{Alexandrino2015lie}) and its more general version in SRF theory. Succinctly, Theorem~\ref{SliceThm} provides relevant information about the local  model of SFFs. However, unlike the Riemannian case, the slice theorem cannot be used directly in the proof of equifocality, because the diffeomorphism in the statement is not the exponential map. It can, nonetheless, be used to demonstrate the stratification structure of SFFs.

\begin{proposition}\cite[Proposition~3.22]{Alexandrino2019singular}
Each stratum is a disjoint union of embedded submanifolds, and the collection of all the strata is a stratification of the manifold in the usual sense. Moreover, the foliation $\f\vert_{\Sigma}$ restricted to a stratum $\Sigma$ is a (regular) Finsler foliation.
\end{proposition}

At last, the stratification structure allows us to create a criterion to determine if $\f$ is a SFF with respect to another Finsler metric.

\begin{proposition}\cite[Proposition~3.23]{Alexandrino2019singular}
\label{SFFstratCriterion}
Suppose there exists a complete Finsler metric $\widetilde F$ on $M$ such that the foliation restricted to each stratum is a (regular) Finsler foliation with respect to $\widetilde{F}$. Then $\f$ is a singular Finsler foliation on $(M, \widetilde F)$.    
\end{proposition}
\section{SFF in \texorpdfstring{$(\alpha,\beta)$}{alpha,beta}-spaces: Proof of Theorem~\ref{AlphaBetaSFF}}

Theorem~\ref{AlphaBetaSFF} will stem from Proposition~\ref{SFFstratCriterion} above, and Propositions~\ref{AlphaBetaSFFHoriz}, \ref{AlphaBetaSFFVert} below. To prove the latter, Proposition~\ref{AlphaBetaSubm} will be  necessary, though not enough. We begin with the following auxiliary result.

\begin{lemma}
\label{LinearTangencyMinStrat}
Let $F = \alpha \phi \big( \tfrac{\beta}{\alpha} \big)$ be an $(\alpha,\beta)$-Minkowski norm on a real $n$-dimensional vector space $V$, and let $\f$ be a singular Finsler foliation of $(V, F)$. Suppose $\phi$ has nowhere vanishing derivative, and the zero vector is a singular leaf; i.e. $L_0 = \{0\}$. Then the vector $\bb$, $\alpha$-dual to the linear functional $\beta$, is tangent to the minimal stratum; i.e. the stratum of dimension zero.
\end{lemma}
\begin{proof}
We will first show that $\bb$ is $\alpha$-horizontal, i.e.
\begin{equation}
\label{B-alpha-horiz}
\ga (\bb , w) = 0 \, \text{ for all } \, w \in T_v L_v \, \text{ and } \, v \in V ,
\end{equation}
where the isomorphism $T_v V \cong V$ is tacitly assumed.

The statement is trivial when $v=0$, because $T_0 L_0 = \{0\}$. For each $v \in V \setminus 0$, the curves $\gamma_v: t \mapsto (t+1)v$ and $\gamma_{-v}: t \mapsto (1-t)v$ are horizontal geodesics, because they pass through the singular leaf $L_0 = \{0\}$. In particular, the Legendre transformations $\ell_v$ and $\ell_{-v}$ vanish on $T_v L_v$. By equation~\eqref{Lambda}, $\lambda(s) := \phi(s) - s\phi^{\prime}(s) > 0$ for all $\vert s \vert \leq \alpha^{\ast}(\beta)$. So it follows from Lemma~\ref{AlphaBetaLegendre} that
\begin{equation*}
\begin{aligned}
\ga \left( \frac{v}{\alpha(v)} + \frac{\phi^{\prime}(r)}{\lambda(r)} \bb , w \right) &= 0 , \\
\ga \left( -\frac{v}{\alpha(v)} + \frac{\phi^{\prime}(-r)}{\lambda(-r)} \bb , w \right) &= 0 ,
\end{aligned}
\end{equation*}
for any $w \in T_v L_v$, and $r := \tfrac{\beta(v)}{\alpha(v)}$. Adding these two equations, we obtain
\begin{equation}
\label{B-alpha-horiz-cond}
\left( \frac{\phi^{\prime}(r)}{\lambda(r)} + \frac{\phi^{\prime}(-r)}{\lambda(-r)} \right) \ga (\bb , w) = 0 .
\end{equation}
When $\phi^{\prime}$ is nonvanishing, we claim that
\begin{equation}
\label{ineqPositive}
\left( \frac{\phi^{\prime}(r)}{\lambda(r)} + \frac{\phi^{\prime}(-r)}{\lambda(-r)} \right) > 0 .
\end{equation}
Indeed, $\phi(s) > s\phi^{\prime}(s)$ and $\phi(-s) > -s\phi^{\prime}(-s)$ for all $\vert s \vert \leq \alpha^{\ast}(\beta)$; recall~\eqref{Lambda}. Hence,
\begin{equation*}
\begin{aligned}
\frac{\phi^{\prime}(r)}{\lambda(r)} + \frac{\phi^{\prime}(-r)}{\lambda(-r)} &= \frac{\lambda(-r)\phi^{\prime}(r) + \lambda(r)\phi^{\prime}(-r)}{\lambda(r) \lambda(-r)} = \frac{\phi(-r)\phi^{\prime}(r) + \phi(r)\phi^{\prime}(-r)}{\lambda(r) \lambda(-r)} \\
&> \frac{-r\phi^{\prime}(-r)\phi^{\prime}(r) + r\phi^{\prime}(r)\phi^{\prime}(-r)}{\lambda(r) \lambda(-r)} = 0 .
\end{aligned}
\end{equation*}
Equations \eqref{B-alpha-horiz-cond} and~\eqref{ineqPositive} imply \eqref{B-alpha-horiz}, just as wanted.

Once $\bb$ is $\alpha$-horizontal,
$$ T_v L_v \subset \{ w \in T_v V : \beta (w) = \ga (\bb,w) = 0 \} .$$
As an element of $\mathfrak{X}(V)$, $\bb$ is a constant vector field. Then, for each $v \in V$,
\begin{equation}
\label{AffineHyperplane}
L_v \subset \mathcal{T}_{v} := \{ y \in V : \beta (y) = \beta (v) \} .
\end{equation}
In other words, $L_v$ is contained in the affine hyperplane $\mathcal{T}_{v}$, which contains $v$ and has $\bb$ as a normal vector.

We can, at last, prove $\bb$ is tangent to the minimal stratum $\Sigma_0$, containing zero.

If $\beta$ is identically null there is nothing to be done. Assume $\bb\neq 0$. By the stratification structure, $\Sigma_0$ is a vector subspace of $V$, so it suffices to show there is some $\delta > 0$ such that $v_{\delta} := \delta \frac{\bb}{\alpha(\bb)} \in \Sigma_0$, i.e. $L_{v_{\delta}} = \{v_{\delta}\}$.

Let $U_0$ be a tubular neighborhood of the singular leaf $L_{0} = \{ 0\}$. Choose $\delta>0$ small enough to consider a future cylinder $C^{+}_{\delta}(0) \subset U_0$. Then $L_{v_{\delta}} \cap U_0 \subset C^{+}_{\delta}(0)$ and, by \eqref{AffineHyperplane}, $L_{v_{\delta}}$ is contained in the affine hyperplane $\mathcal{T}_{v_{\delta}}$. Since $C_{\delta}^{+}(0)$ is invariant by the rotations of $(V,\alpha)$ fixing $\bb$ (recall Corollary~\ref{AlphaBetaIsometry}), the hyperplane $\mathcal{T}_{v_{\delta}}$ is tangent to $C_{\delta}^{+}(0)$ at $v_{\delta}$. Thus,
\begin{equation*}
L_{v_{\delta}} \cap U_0 \subset C_{\delta}^{+}(0) \cap \mathcal{T}_{v_{\delta}} = \{v_{\delta}\} ,
\end{equation*}
and the connectedness of leaves finishes the proof.
\end{proof}

As mentioned in Remark~\ref{TangencyToStrata}, by Lemma~\ref{LinearTangencyMinStrat} and the Slice Theorem~\ref{SliceThm}, if $\phi$ has nowhere zero derivative, then $\bb$ is tangent to the stratum. Proposition~\ref{AlphaBetaSFFHoriz} now follows directly from Proposition~\ref{AlphaBetaSubm}.

\begin{proposition}
\label{AlphaBetaSFFHoriz}
Let $\f$ be a SFF on the $(\alpha,\beta)$-Finsler manifold $(M,F)$, where $F = \alpha \phi \big( \tfrac{\beta}{\alpha} \big)$. Suppose $\phi$ has nowhere zero derivative and $\bb$ ($\alpha$-dual to $\beta$) is $\alpha$-orthogonal to the leaves of $\f$. Then $\bb$ is foliated and horizontal, and $\f$ restricted to each stratum is a Riemannian foliation with respect to the norm $\alpha$.
\end{proposition}

Next, the proof of Proposition~\ref{AlphaBetaSFFVert} is neither long nor involving, but it does have some subtlety.

\begin{proposition}
\label{AlphaBetaSFFVert}
Let $\f$ be a SFF on the $(\alpha,\beta)$-Finsler manifold $(M,F)$, where $F = \alpha \phi \big( \tfrac{\beta}{\alpha} \big)$. Assume $\bb$ ($\alpha$-dual to $\beta$) is a nowhere vector field tangent to the leaves of $\f$. There exist a smooth function $\kappa: M \to \mathbb{R}$ such that $\f$ restricted to each stratum is a Riemannian foliation with respect to the Riemannian metric $\kappa\ga$, where $\alpha(\cdot) = \sqrt{\ga(\cdot,\cdot)}$.
\end{proposition}
\begin{proof}
In the proof of Proposition~\ref{AlphaBetaSubm} item~\ref{s3}, we have seen that the (intrinsic) function $h: M \to \mathbb{R}$ given by
\begin{equation}
\label{equation-definition-h-sff}
h(p) := \frac{\beta(v)}{\alpha(v)} \, , \, \text{ for any } \, v \in \nu_{p}(L_p) , 
\end{equation}
is well-defined; recall \eqref{BetaOverAlphaFunct}. Also due to Proposition~\ref{AlphaBetaSubm}, we know the function $\kappa: M \to \mathbb{R}$ is defined as
$$ \kappa = \rho \circ h ,$$
where $\rho(s) = \phi(s) (\phi(s) - s\phi^{\prime}(s))$, and it is enough to check that $h$ is smooth.

Consider a (possibly singular) point $q \in M$ and a tubular neighborhood $U_q$ of a plaque $P_q$. The definition of singular foliation allow us to choose a set of vector fields $\{X_k\}$ tangent to $\f$ so that $\{X_k(q)\} \cup \{\bb(q)\}$ is a basis  of $T_{q}L_q$. Reducing  the neighborhood $U_q$ if necessary, we may assume the vectors $\{X_k(p)\} \cup \{\bb(p)\}$ are linearly independent for all $p \in U_q$. Let $\mathcal{V}$ be the distribution spanned by the vector fields $\{X_k\}$; in particular, $\mathcal{V}(q) = T_q L_q$.

By construction,
\begin{equation}
\label{equation-normalV-normalL-Bvertical}
\nu_p (L_p) \subset \nu_p (\mathcal{V}(p)) \, \text{ and } \, \bb(p) \in \mathcal{V}(p) \, \text{ for all } p \in U_q .   
\end{equation}
Equations~\eqref{equation-definition-h-sff}, \eqref{equation-normalV-normalL-Bvertical}, and Lemma~\ref{BetaOverAlphaCte} imply
\begin{equation}
\label{equation-h-V-SFF}
h(p) = \frac{\beta(v)}{\alpha(v)} \, \text{ for all } v \in \nu_p (\mathcal{V}(p)) .
\end{equation}

Finally, consider a vector field $Y \in \mathfrak{X}(U_q)$ so that  
\begin{equation}
\label{equation-vectorY-SFF}
Y(p) \in \nu^{1}_{p}(\mathcal{V}(p))  
\end{equation}
Equations~\eqref{equation-h-V-SFF}, and~\eqref{equation-vectorY-SFF} imply
$$ h(p) = \frac{\beta(Y(p))}{\alpha(Y(p))} ,$$
which concludes the proof that $h$ is smooth. 
\end{proof}

Now, we are ready to prove the theorem.

\begin{proof}[Proof of Theorem~\ref{AlphaBetaSFF}]
From Proposition~\ref{AlphaBetaSFFHoriz} or Proposition~\ref{AlphaBetaSFFVert}, the restriction of $\f$ to each stratum is a Riemannian foliation with respect to the metric $\ga$ or $\kappa \ga$, respectively. Thus, Proposition~\ref{SFFstratCriterion} affirms $\f$ is a SRF on $M$ with respect to $\ga$ or $\kappa \ga$, respectively.
\end{proof}

\begin{remark}
The proof of Remark~\ref{AlphaBetaConverse} is analogous to the argument above, taking into account that a foliated horizontal vector field $\bb$ of a SRF $\f$ must automatically be tangent to the strata.
\end{remark}
\section{Equifocality: Proof of Theorem~\ref{FinslerEquifocal}}

Let us recall the hypotheses of Theorem~\ref{FinslerEquifocal}. On a complete Finsler manifold $(M,F)$, we have a singular Finsler foliation (SFF) $\f = \{ L \}$ such that:
\begin{enumerate}[label=\textup{(\roman*)}]
\item $\f$ is a singular Riemannian foliation (SRF) with respect to a complete Riemannian norm $\alpha$; and

\item $\f$ has (locally) closed leaves.
\end{enumerate}
We summarize the main idea of the theorem's proof in the following lemma.

\begin{lemma}
\label{EquifocalMain}
Under the hypotheses of Theorem~\ref{FinslerEquifocal}, let $U_q$ be a trivializing neighborhood of some plaque $P_q \subset L_q$, as defined in Remark~\ref{TrivNbhd}. Consider a family of unit speed horizontal segments of geodesics $\gamma_{\theta}: [-\delta,\delta] \to U_q \subset (M,F)$, with $\theta \in [-\theta_0,\theta_0]$ for some $\theta_0 > 0$, such that:
\begin{enumerate}[label=\textup{(\roman*)}]
\item $\gamma_{\theta} (-\delta) \in L_{\gamma_{0}(-\delta)}$ for all $\theta$;

\item $q_{\theta} := \gamma_{\theta}(0) \in L_{q}$ for all $\theta$; and

\item $q_{0} = q$ is the only (possibly) singular point on $\gamma_{0}\vert_{[-\delta,\delta]}$, i.e. $L_q$ is may not be in the regular stratum.
\end{enumerate}
Then $\gamma_{\theta}(r) \in L_{\gamma_{0}(r)}$ for all $\theta$ and for $r>0$ small enough.
\end{lemma}
\begin{proof}
We start by defining several mathematical objects. Without loss of generality (reducing $\delta> 0$ if necessary), assume $\exp^{F}_{q_{\theta}}(B_{\delta}(0))$, $\exp^{F(-)}_{q_{\theta}}(B_{\delta}(0))$, and $\exp^{\alpha}_{q_{\theta}}(B_{\delta}(0))$ are subsets of $U_q$ for all $\theta$.
\begin{enumerate}[label=\textbf{\roman*.}]
\item $S_{q_{\theta}}^{+} := \exp_{q_\theta}^{F}\big(\ \nu_{q_\theta}^{F} L \cap B_{\delta}(0) \big);$

\item $S_{q_{\theta}}^{-} :=\exp_{q_\theta}^{F(-)} \big(\nu_{q_\theta}^{F(-)} L \cap B_{\delta}(0) \big);$

\item $S_{q_\theta} := \exp_{q_\theta}^{\alpha}\big( \nu_{q_\theta}^{\alpha} L \cap B_{\delta}(0) \big);$

\item $v_{\theta}:=\gamma^{\prime}_{\theta}(0);$

\item $\widehat{\gamma}_{\theta}(r) := \exp_{q_\theta}^{F(-)}(-r v_{\theta})$, where $r>0$.
\end{enumerate}
Note that $S_{q_{\theta}}^{+}, S_{q_{\theta}}^{-}, S_{q_{\theta}} \subset U_q$ for all $\theta$, and $\widehat{\gamma}_{\theta}(r)=\gamma_{\theta}(-r)$ for all $r \in (0,\delta]$.

We also need to define two families of foliated diffeomorphisms.
\begin{enumerate}[label=\textbf{\roman*.},start=6]
\item $\varphi_{\theta}^{-}: \big( S_{q_\theta}^{-} \setminus \{ q_\theta \} , S_{q_\theta}^{-} \cap \f \big) \to \big( S_{q_\theta} \setminus \{ q_\theta \} , S_{q_\theta} \cap \f \big)$;

\item $\varphi_{\theta}^{+}: \big( S_{q_\theta} \setminus \{ q_\theta \} , S_{q_\theta} \cap \f \big) \to \big( S_{q_\theta}^{+} \setminus \{ q_\theta \} , S_{q_\theta}^{+} \cap \f \big)$.
\end{enumerate}
These maps will extend to $q_{\theta}$ as homeomorphisms. For their constructions, recall that, by Remark~\ref{TrivNbhd}, there exists a regular foliation $\f^{\pi}$ on $U_q$ such that:
\begin{itemize}
\item $\f^{\pi}$ is a subfoliation of $\f_{U_q}:= \f \cap U_q$, i.e. the leaves of $\f^{\pi} = \{ L^{\pi} \}$ are contained in the leaves of $\f_{U_q}$; and

\item the leaf $L_{q}^{\pi}$ is an open subset of $L_q$.
\end{itemize}
From the construction of $\f^{\pi}$, for each $x \in S_{q_\theta}$, we have
\begin{equation*}
L^{\pi}_x \cap S_{q_\theta}^{-} = \{ x^{-} \} , \, \text{ and } \, L^{\pi}_x \cap S_{q_\theta}^{+} = \{ x^{+} \} .
\end{equation*}
So we define the foliated diffeomorphisms as:
\begin{equation}
\label{EquifocalEq1Diffeos}
\varphi_{\theta}^{-} (x^{-}) = x , \, \text{ and } \, \varphi_{\theta}^{+} (x) = x^{+} .
\end{equation}

The proof will be divided into 3 steps.

\begin{step}
\label{step1}
Consider the variation $\theta \mapsto \varphi_{\theta}^{-} \circ \widehat{\gamma}_{\theta}$. It has the interesting property that, for each $r \in (0,\delta]$, the points $\varphi_{\theta}^{-} \circ \widehat{\gamma}_{\theta} (r)$ lie in the same leaf of $\f$ for all $\theta$, because $\varphi^{-}_{\theta}$ are foliated diffeomorphisms and, by the hypotheses, $L_{\widehat{\gamma}_{\theta} (r)} = L_{\widehat{\gamma}_{0} (r)}$ for all $\theta$ and each $r \in (0,\delta]$. However, the curves in this variation are not necessarily geodesics with respect to $\alpha$. Therefore, we need to consider a variation of $\alpha$-geodesics. For $v_{\theta}^{0} := - (\varphi_{\theta}^{-} \circ \widehat{\gamma}_{\theta})^{\prime} (0)$, we define the variation
\begin{equation*}
\theta \mapsto \gamma_{\theta}^{0} : r \mapsto \exp_{q_\theta}^{\alpha} \left( - r \frac{v_{\theta}^{0}}{\alpha(v_{\theta}^{0})} \right) .
\end{equation*}
We will prove, vide Equation~\eqref{EquifocalEq5}, the new variation still has the property that $\theta \mapsto \gamma_{\theta}^{0} (r)$ always lies in the same leaf of $\f$ (for each $r \in (0,\delta]$); see Figure~\ref{EquifocalFig1}.
\end{step}

\begin{figure}
\centering
\includegraphics[scale=0.5]{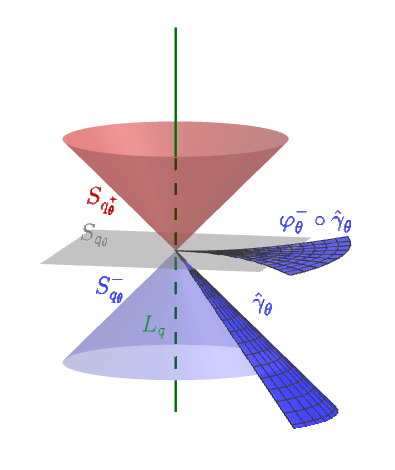}
\caption{Diagram illustrating Step~\ref{step1} of the proof of Lemma~\ref{EquifocalMain} in the particular case when $q_{\theta}=q$.}
\label{EquifocalFig1}
\end{figure}

We first claim that
\begin{equation*}
- v_{\theta}^{0} = (\varphi_{\theta}^{-} \circ \widehat{\gamma}_{\theta})^{\prime} (0)
\end{equation*}
is the orthogonal projection of $- v_{\theta}$ onto $\nu^{\alpha}_{q_\theta} L$ with respect to the Riemannian norm $\alpha$. In order to verify it, consider the projection $\pi: U_q \to S_q$ that describes the leaves of $\f^\pi$, i.e. $L^{\pi}_p = \pi^{-1}(s)$, where $s=\pi(p)$. For each $\theta$, $\pi$ restricts to a projection from $\exp_{q_{\theta}}^{F(-)} (B_{\delta}(0))$ onto $S_{q_{\theta}}$. Moreover, $\pi\vert_{S_{q_\theta}^{-}} = \varphi^{-}_{\theta}$. So the $\alpha$-orthogonal projection of $- v_{\theta}$ onto $\nu^{\alpha}_{q_\theta} L$ is $d \pi (- v_{\theta}) = (\pi \circ \widehat{\gamma}_{\theta})^{\prime} (0) = (\varphi_{\theta}^{-} \circ \widehat{\gamma}_{\theta})^{\prime} (0) = -v^0_{\theta}$, as wanted.

Next, we construct a family of unit speed segments of geodesics $\gamma_{\theta}^{\varepsilon}: [0,\delta] \to (S_{q_\theta}, \alpha)$ for $0 < \varepsilon \leq \delta$ by requiring that $\gamma_{\theta}^{\varepsilon}(0) = q_{\theta}$ and
\begin{equation}
\label{EquifocalEq2}
\gamma_{\theta}^{\varepsilon} (r^{\varepsilon}_{\theta}) =     \varphi_{\theta}^{-} \circ \widehat{\gamma}_{\theta} (\varepsilon) \, \text{ for some } \, r^{\varepsilon}_{\theta} \in (0,\delta] .
\end{equation}
They are well-defined because $\exp_{q_\theta}^{\alpha}:(\nu_{q_\theta}^{\alpha}L\cap B_{\delta}(0))\to S_{q_\theta}$ is  a bijective mapping. The definition of $\varphi_{\theta}^{-}$ in \eqref{EquifocalEq1Diffeos} gives that
\begin{equation}
\label{EquifocalEq3}
\varphi_{\theta}^{-} \circ \widehat{\gamma}_{\theta}(\varepsilon) \in L_{\widehat{\gamma}_{\theta}(\varepsilon)} \cap S_{q_\theta} .
\end{equation}
The hypotheses of the lemma assured us that
\begin{equation}
\label{EquifocalEq4} 
L_{\widehat{\gamma}_{\theta} (\varepsilon)} = L_{\widehat{\gamma}_{0}(\varepsilon)} \, \text{ for all } \, \theta .
\end{equation}
Equations~\eqref{EquifocalEq2}, \eqref{EquifocalEq3}, and~\eqref{EquifocalEq4} imply that $r^{\varepsilon} := r_{\theta}^{\varepsilon}$ is independent of $\theta$, and
\begin{equation*}
\gamma^{\varepsilon}_{\theta} (r^{\varepsilon}) \in L_{\widehat{\gamma}_{0}(\varepsilon)} \, \text{ for all } \, \theta . 
\end{equation*}
The above equation and the Homothetic Transformation Lemma~\ref{HomotheticLemma} guarantee the existence of some $x^{\varepsilon,r} \in C_{r}^{\alpha}(P_q)$, the cylinder of radius $r \in (0,\delta]$ around the plaque $P_q$ with respect to the Riemannian metric $\alpha$, such that
\begin{equation*}
\gamma_{\theta}^{\varepsilon} (r) \in L_{x^{\varepsilon, r}} \, \text{ for all } \, \theta .
\end{equation*}
Once $\displaystyle \lim_{\varepsilon \to 0}(\gamma_{\theta}^{\varepsilon})^{\prime}(0) = - \tfrac{v_{\theta}^{0}}{\alpha(v_{\theta}^{0})}$, and $\f$ is locally closed, it follows that
\begin{equation}
\label{EquifocalEq5}
\gamma_{\theta}^{0}(r)\in L_{\gamma_{0}^{0}(r)} \, \text{ for all } \, \theta ,
\end{equation}
where $r \mapsto \gamma_{\theta}^{0}(r):= \exp_{q_\theta}^{\alpha} \left( - r \frac{v_{\theta}^{0}}{\alpha(v_{\theta}^{0})} \right)$ for $r\in (0,\delta]$.

\begin{step}
\label{step2}
We now consider the reverse Riemannian geodesics $r \mapsto \gamma^{\alpha}_{\theta} (r) := \gamma^0_{\theta}(-r)$ for $r \in [0,\delta]$ and check that, for each $r \in (0,\delta]$, $\theta \mapsto \gamma_{\theta}^{\alpha}(r)$ is always contained in the same leaf of $\f$; see Figure~\ref{EquifocalFig2}.
\end{step}

\begin{figure}
\centering
\includegraphics[scale=0.5]{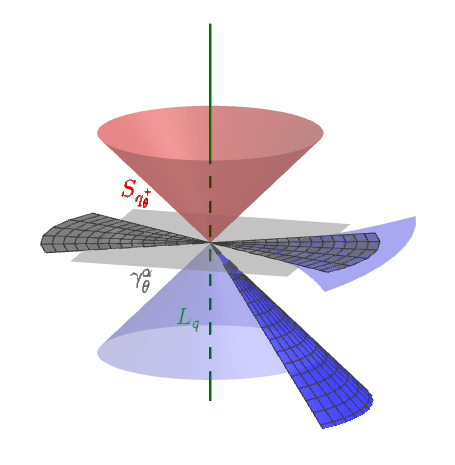}
\caption{Diagram illustrating Step~\ref{step2} of the proof of Lemma~\ref{EquifocalMain} in the particular case when $q_{\theta}=q$.}
\label{EquifocalFig2}
\end{figure}

Indeed, $\f$ is a singular Riemannian foliation (SRF) with respect to the Riemannian norm $\alpha$, and regular leaves of SRFs are equifocal. So:
\begin{equation}
\label{EquifocalEq6}
\gamma_{\theta}^{\alpha}(r)= \exp_{q_\theta}^{\alpha} \left( r \frac{v_{\theta}^{0}}{\alpha(v_{\theta}^{0})} \right) \in L_{\gamma_{0}^{\alpha}(r)} \cap S_{q_{\theta}} .
\end{equation}

\begin{step}
\label{step3}
We can proceed to our main goal: showing that
\begin{equation}
\label{EquifocalEq7}
\gamma_{\theta}(r) = \exp_{q_\theta}^{F}(r v_{\theta}) \in L_{\gamma_{0}(r)} , \, \text{ for all } \theta \text{ and small } r>0 .
\end{equation}
Here, the arguments and strategies are similar to those of Step~\ref{step1}; see Figure~\ref{EquifocalFig3}.
\end{step}

\begin{figure}[ht]
\centering
\includegraphics[scale=0.5]{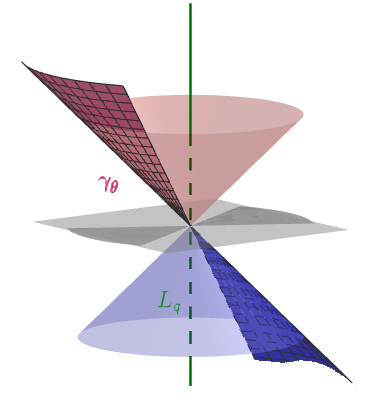}
\caption{Diagram illustrating Step~\ref{step3} of the proof of Lemma~\ref{EquifocalMain} in the particular case when $q_{\theta}=q$.}
\label{EquifocalFig3}
\end{figure}

Define a family of unit speed  segments of geodesics $\gamma_{\theta}^{+,\varepsilon}: [0,\delta] \to (S_{q_\theta}^{+},F)$ for $0 < \varepsilon \leq \delta$ by the conditions $\gamma_{\theta}^{+,\varepsilon}(0) = q_{\theta}$ and
\begin{equation}
\label{EquifocalEq8}
\gamma_{\theta}^{+,\varepsilon}(\tilde{r}^\varepsilon_{\theta}) =     \varphi_{\theta}^{+}\circ\gamma_{\theta}^{\alpha}(\varepsilon)  \, \text{ for some } \, \tilde{r}^{\varepsilon}_{\theta} \in (0,\delta] .
\end{equation}
From the definition of $\varphi_{\theta}^{+}$ in \eqref{EquifocalEq1Diffeos}, we have
\begin{equation}
\label{EquifocalEq9}
\varphi_{\theta}^{+}\circ\gamma_{\theta}^{\alpha}(\varepsilon)\in L_{\gamma_{\theta}^{\alpha}(\varepsilon)} \cap S_{q_\theta}^{+} .
\end{equation}
Equations~\eqref{EquifocalEq8}, \eqref{EquifocalEq9}, and \eqref{EquifocalEq6} imply that $\tilde{r}^{\varepsilon}_{\theta} := \tilde{r}^{\varepsilon}$ is independent of $\theta$, and
\begin{equation*}
\gamma_{\theta}^{+,\varepsilon}(\tilde{r}^\varepsilon)\in L_{\gamma_{0}^{\alpha}(\varepsilon)} \, \text{ for all } \, \theta .
\end{equation*}
The above equation and the Homothetic Transformation Lemma~\ref{HomotheticLemma} now garantee the existence of $x^{\varepsilon,\tilde{r}} \in C_{\tilde{r}}^{F}(P_{q})$, the cylinder of radius $\tilde r \in (0,\delta]$ around the plaque $P_q$ with respect to the Finsler metric $F$, such that
\begin{equation}
\label{EquifocalEq10}
\gamma_{\theta}^{+,\varepsilon}(\tilde{r})\in L_{x^{\varepsilon,\tilde{r}}} .
\end{equation}

We claim that
\begin{equation}
\label{EquifocalEq11}
\lim_{\varepsilon\to 0} (\gamma_{\theta}^{+,\varepsilon})^{\prime} (0) = v_{\theta} = \gamma_{\theta}^{\prime} (0) .
\end{equation}
Equations~\eqref{EquifocalEq10}, \eqref{EquifocalEq11}, and the fact that $\f$ is locally closed will imply the goal \eqref{EquifocalEq7}, concluding the proof. So it only remains to prove \eqref{EquifocalEq11}.

By contradiction, suppose \eqref{EquifocalEq11} fails. Then, by compactness, there exists a subsequence $\varepsilon_{n} \to 0$ such that
\begin{equation}
\label{EquifocalEq12}
\lim_{n\to 0} (\gamma_{\theta}^{+,\varepsilon_n})^{\prime} (0 )=: w_{\theta} \neq v_{\theta} ,
\end{equation}
where $w_\theta \in \nu_{q_\theta}^{F} L$, and $F(w_{\theta}) = 1$. Since $\f^\pi$ is a regular foliation, we can assume (using a trivialization of the foliation) that the leaves of $\f^\pi$ are preimages of the canonical linear submersion $\pi: \mathbb{R}^{l+k} \to \mathbb{R}^k$, defined as $ \pi(x,y) = x$, i.e. $L_{y}^{\pi} = \pi^{-1}(x)$. Using a diffeomorphism on $\mathbb{R}^{k}$, it is also possible to assume that $\pi \circ \gamma_{\theta}^{\alpha}$ is the line $t \mapsto \sigma(t) := t \tfrac{v_{\theta}^{0}}{\alpha(v_{\theta}^{0})}$. Consider the $\f^{\pi}$-saturation of $\sigma$, i.e. the subspace
\begin{equation}
\label{EquifocalEq13}
\Sigma := \bigcup_{y\in\sigma} L^{\pi}_{y} = \pi^{-1}(\sigma).  
\end{equation}
On one hand, the definition of $\varphi_{\theta}^{+}$ in \eqref{EquifocalEq1Diffeos} gives that
\begin{equation}
\label{EquifocalEq14}
\varphi^{+}_{q_\theta} \circ \gamma_{\theta}^{\alpha} \subset \Sigma \cap S_{q_\theta}^{+} .   
\end{equation}
At the same time,
\begin{equation}
\label{EquifocalEq15}
w_{\theta} \notin T_{q_\theta} \Sigma ;
\end{equation}
otherwise, its $\alpha$-projection $w^0_{\theta}$ would be a multiple of $v^{0}_{\theta}$ by \eqref{EquifocalEq13}, which would mean that the $F$-unit vector $w_{\theta}$ should be a multiple of the $F$-unit vector $v_{\theta}$, contradicting \eqref{EquifocalEq12}. From \eqref{EquifocalEq15}, there must exist $N, \tilde \delta> 0$ such that
\begin{equation}
\label{EquifocalEq16}
\gamma_{\theta}^{+,\varepsilon_n} (\tilde r) \notin \Sigma \, \text{ for } \, n > N \, \text{ and } 0 < \tilde r < \tilde \delta .
\end{equation}
Equations~\eqref{EquifocalEq14} and~\eqref{EquifocalEq16} contradict the construction of $\gamma_{\theta}^{+,\varepsilon}$; recall~\eqref{EquifocalEq8}. Equation~\eqref{EquifocalEq11} is thusly verified.
\end{proof}

Call to mind the concept of endpoint map from Definition~\ref{DefEquifocal}, i.e. the map $\eta_{r\xi}: P_q \to M$, for a plaque $P_q$ of a regular leaf, a foliated horizontal unit vector field $\xi$, and $r > 0$, which is simply given by $\eta_{r\xi}(x) = \exp_x (r\xi)$. The next result basically renders Lemma~\ref{EquifocalMain} into properties of endpoint maps.

\begin{lemma}
\label{EquifocalLemma}
Under the hypotheses of Theorem~\ref{FinslerEquifocal}, let $\gamma: \mathbb{R} \to (M,F)$ be a unit speed horizontal geodesic starting at regular leaf $L_{q}$; to be precise, $\gamma(0) = q$. For each $t_0 > 0$, there exists a neighborhood $U$ of $\gamma(t_0)$ in $M$ and $\delta > 0$ such that:
\begin{enumerate}[label=\textup{(\roman*)}]
\item\label{e1} $\eta_{r\xi} (P_{\gamma(t)}) \subset L_{\gamma(t+r)}$ for all $t \in (t_0 - \delta , t_0)$, $t+r \in (t_0 - \delta, t_0 + \delta)$ with $r>0$, where $P_{\gamma(t)}$ is the connected component of $L_{\gamma(t)} \cap U$ that contains $\gamma(t)$;

\item\label{e2} $d\eta_{r\xi}: TP_{\gamma(t)} \to TM$ has constant rank.
\end{enumerate}
\end{lemma}
\begin{proof}
Item~\ref{e1} follows directly from Lemma~\ref{EquifocalMain}.

To prove item~\ref{e2}, reduce the neighborhood $U$, if necessary, so that $L_{\gamma(t_0)}$ is the only (possibly) singular leaf among the leaves $L_{\gamma(t)}$ for $t \in (t_0 - \delta, t_0 )$. The existence of such neighborhood also follows from Lemma~\ref{EquifocalMain}.

Now we consider three cases.

First, suppose $r>0$ satisfies $t_0 - \delta < t + r < t_{0}$. In this case, $L_{\gamma (t + r)}$ is a regular leaf, and it follows from the theory of regular Finsler foliations that $\eta_{r\xi}:  P_{\gamma(t)} \to M$ is a diffeomorphism onto its image $L_{\gamma(t+r)}$. Hence, the dimension of the rank of $d \eta_{r\xi}: T P_{\gamma(t)} \to T M$ is always the dimension of the regular leaves.

Next, assume that $t + r = t_0$. From the past tubular structure of $L_{\gamma(t_0)}$, the map $\eta_{r\xi}:  P_{\gamma(t)} \to L_{\gamma(t_0)}$ is a submersion. Thus, the dimension of the rank of $d \eta_{r\xi}: T P_{\gamma(t)} \to T M$ is constant equal to the dimension of $L_{\gamma(t_0)}$.

Otherwise, $t_0 < t + r < t_0 + \delta$. As before, everywhere on the plaque, the dimension of the rank of $d \eta_{r\xi}: T P_{\gamma(t)} \to T M$ is the dimension of the regular leaves, by the claim that $\eta_{r\xi}: P_{\gamma(t)} \to M$ is again a diffeomorphism onto its image. To verify this fact, we begin by noticing that the future metric projection onto $P_{\gamma(t_0)}$ is a submersion. Then $\ker d\eta_{r \xi}$ is contained in the past slice $S^{-}_{\gamma(t_0)}$. But $\eta_{r\xi} \vert_{S_{\gamma(t_0)}^{-} \cap L_{\gamma(t)}}$ is a diffeomorphism onto its image, due to the absence of a conjugate point, with the (possible) exception of $\gamma(t_0)$. So $\dim \ker d\eta_{r\xi}=0$, and $\eta_{r\xi}:  P_{\gamma(t)} \to L_{\gamma(t+r)}$ is indeed a diffeomorphism.
\end{proof}

Finally, we can prove the equifocality of a regular leaf $L_q$.

\begin{proof}[Proof of Theorem~\ref{FinslerEquifocal}]

Let $P_q$ be the largest neighborhood in $L_q$ around $q$ such that the horizontal bundle $\nu L_q$ is trivial. In order to see it can be achieved, recall that the regular stratum is an open subset of $M$. Taking $U_{q}$ as the largest neighborhood of $q$ in which $\f\vert_{U_q}$ is described by a submersion, we define $P_q$ as the connected component of $ L_q\cap U_q$ containing $q$.

Consider the unit horizontal geodesic $t \mapsto \gamma(t) := \eta_{t\xi}(q)$. We employ the exact same argument presented in \cite[\S~3]{Alexandrino2008equifocality}. Roughly speaking, the proof follows by composing the ``small" endpoint maps from Lemma~\ref{EquifocalLemma} along the geodesic $\gamma$. For didactic purposes, let us briefly recall the argument.

Fixed $r>0$, let $\{U_i\}_{i=0}^{n}$ be a finite open covering of $\gamma_{[0,r]}$ by neighborhoods defined in Lemma~\ref{EquifocalLemma}; namely, $U_i$ is a neighborhood of $\gamma(t_i)$ satisfying Lemma~\ref{EquifocalLemma}, where $0 = t_0 < \cdots < t_n = r$. Then, using the lemma, we can find:
\begin{enumerate}[label=\textbf{\roman*.}]
\item a partition $\{ \delta_i \}_{i=1}^{n}$ of $[0,r]$ such that $t_{i-1} < \delta_i < t_i$ for $1 \leq i \leq n$;

\item open sets $P_i \subset L_{\gamma(\delta_i)}$ for $1 \leq i \leq n$, $P_{0} \subset L_q$, and $P_{n+1} \subset L_{\gamma(r)}$;

\item unit horizontal foliated vector fields $\xi_i$ along $P_{i}$ for $0 \leq i \leq n$; and

\item positive real numbers $\{r_{i}\}_{i=0}^{n}$;
\end{enumerate}
with the following properties:
\begin{itemize}
\item $P_{i} \subset U_{i-1} \cap U_{i}$ for $1 \leq i \leq n$;

\item  $\xi_i$ is tangent to $\gamma$ for $0 \leq i \leq n$;

\item $\eta_{r_i\xi_i}: P_{i} \to P_{i+1}$ is surjective for $0 \leq i \leq n$, and a diffeomorphism for $0 \leq i < n$;

\item $\eta_{r\xi} = \eta_{r_n\xi_{n}} \circ \eta_{r_{n-1}\xi_{n-1}} \circ \cdots \circ \eta_{r_{0}\xi_{0}}$.
\end{itemize}

The above construction ensures that the application $\eta_{r\xi}:P_{0}\to M$ meets the conditions of Definition~\ref{DefEquifocal}. By a connectedness argument, it is possible to show that the application $\eta_{r\xi}:P_q \to M$ also satisfies the conditions of Definition~\ref{DefEquifocal}; see \cite{Alexandrino2008equifocality} for more details. This concludes the proof.
\end{proof}
\section{Closure of a SFF: Proof of Proposition~\ref{MolinoSFF} }

\begin{proof}[Proof of Proposition~\ref{MolinoSFF}]

Since $\f$ is a singular Riemannian foliation (SRF) with respect to the Riemannian metric $\ga$ by assumption, it follows directly from \cite{Alexandrino2017closure} that $\overline{\f} = \{ \overline{L}_p \}_{p\in M}$ is a SRF. This is, in fact, the most complicated part of the demonstration.

Let $\Sigma$ be a fixed stratum of $\overline{\f}$. By Proposition~\ref{SFFstratCriterion}, it suffices to show $\overline{\f}\vert_{\Sigma}$ is a Finsler foliation. However, $\f\vert_{\Sigma}$ is a (regular) Finsler foliation on $\Sigma$, and as such it is locally described by Finsler submersions. Let $\pi: (U_1,F) \to (U_2,F_2)$ be a Finsler submersion that describes $\f\vert_{\Sigma}$ on a neighborhood $U_1 \subset M \cap \Sigma$. Now, it is enough to verify that $\overline{\f}_{U_1} := \overline{\f} \cap U_{1}$ is a Finsler foliation (SFF).

From the classical theory of foliations, if $q \in U_1$, then $\pi \big( L_q \cap U_1 \big)$ is an orbit of the pseudogroup of holonomy  $\mathrm{Hol}(\f\vert_{\Sigma})$ through $s := \pi(q)$. Since $\f\vert_{\Sigma}$ is a Finsler foliation, $\mathrm{Hol}(\f\vert_{\Sigma})$ is a pseudogroup of isometries for $(U_2,F_2)$.

Consider the average Riemannian metric on $U_2$, i.e. for each $s \in U_2$ and $u,w \in T_s U_2$,
\begin{equation*}
\g (u,w) := \int_{\mathcal{I}_s} \g_{v}^{F_2}(u,w)\, \vol_v ,
\end{equation*}
where $\mathcal{I}_s$ is the indicatrix of $F_2$ at $s$, and $\vol_v$ is the Riemannian volume on $T_v \mathcal{I}_s$ induced by the Riemannian metric $\g^{F_2}_v$. The average Riemannian metric $\g$ is invariant by isometries of $(U_2,F_2)$. Indeed, since $\g^{F_2}_v (u,w) = \g^{F_2}_{d\sigma_s(v)} (d\sigma_s(u), d\sigma_s(w))$ for any isometry $\sigma$ of $(U_2, F_2)$, we have that $\g (d\sigma_s(u), d\sigma_s(w)) = \g (u,w)$. In consequence, the pseudogroup $\mathrm{Hol}(\f\vert_{\Sigma})$ turns out to be a pseudogroup of isometries for $(U_2,\g)$. The fact that this transversal metric $\g$ is invariant under $\mathrm{Hol}(\f\vert_{\Sigma})$ makes it is possible (through standard arguments involving \emph{foliated atlases}) to construct a bundle like metric. As a result, we get a Riemannian metric for which $\f\vert_{\Sigma}$ is a Riemannian foliation.

From Salem \cite[Appendix~D, Proposition~2.3 and Theorem~3.1]{Molino1988riemannian}, we infer that $\overline{\mathrm{Hol}(\f\vert_{\Sigma})}$ is a Lie pseudogroup of isometries for $(U_2,\g)$. This allow us to find a set of vector fields $\{X_{i}\}$ whose orbits coincide with the orbits of the action of $\overline{\mathrm{Hol}(\f\vert_{\Sigma})}$ on $U_2$. For a dense set of parameters $t$, the flow $\mathrm{e}^{t X_{i}} \in \mathrm{Hol}(\f)$. Once $\mathrm{Hol}(\f)$ is a pseudogroup of isometries for $(U_2,F_2)$, we obtain that $\mathrm{e}^{t X_{i}}$ is an isometry of $(U_2,F_2)$ for every small $t$ (where the flow can be defined). So $\overline{\mathrm{Hol}(\f\vert_{\Sigma})}$ is a Lie pseudogroup of isometries for $(U_2,F_2)$.  Therefore, the partition of $U_2$ into orbits of the action of $\overline{\mathrm{Hol}(\f\vert_{\Sigma})}$ is a (regular) Finsler foliation of $(U_2,F_2)$. Together with the fact that $\pi:(U_1,F)\to (U_2,F_2)$ is a Finsler submersion, we conclude that $\overline{\f}_{U_1}$ is a Finsler foliation.
\end{proof}

\bibliography{references}

\end{document}